\documentclass[a4paper, 12pt]{article}

\usepackage{amsfonts, amsmath, amsthm}

\usepackage{latexsym}

\theoremstyle{plain}
\newtheorem{thm}{Theorem}[section]
\newtheorem{lem}[thm]{Lemma}
\newtheorem{cor}[thm]{Corollary}

\theoremstyle{definition}
\newtheorem{defn}[thm]{Definition}

\theoremstyle{remark}
\newtheorem{rem}[thm]{Remark}

\numberwithin{equation}{section}

\newcommand{\Z}{\mathbb{Z}}

\newcommand{\R}{\mathbb{R}}
\newcommand{\C}{\mathbb{C}}

\newcommand{\pa}{\partial}
\newcommand{\eps}{\varepsilon}

\newcommand{\jb}[1]{\langle #1 \rangle}
\newcommand{\Jb}[1]{\bigl\langle #1 \bigr\rangle}
\newcommand{\JB}[1]{\Bigl\langle #1 \Bigr\rangle}
\newcommand{\jbf}[1]{\left\langle #1 \right\rangle}

\DeclareMathOperator{\realpart}{\rm Re}
\DeclareMathOperator{\imagpart}{\rm Im}

\newcommand{\dis}{\displaystyle}
\newcommand{\sgn}[1]{\,{\rm sgn}(#1)}

\newcommand{\Co}[3]{C_{#1,#2,#3}}

\begin{document}
\title{
 Null structure in a system of quadratic \\
derivative nonlinear Schr\"odinger equations 
 }  

\author{
         Masahiro Ikeda\thanks{
             Department of Mathematics, Graduate School of Science, 
             Osaka University. 
             1-1 Machikaneyama-cho, Toyonaka, Osaka, 560-0043, Japan 
              (E-mail: {\tt m-ikeda@cr.math.sci.osaka-u.ac.jp})}
         \and  
         Soichiro Katayama\thanks{
             Department of Mathematics, Wakayama University. 
             930 Sakaedani, Wakayama 640-8510, Japan
              (E-mail: {\tt katayama@center.wakayama-u.ac.jp})}
         \and  
         Hideaki Sunagawa\thanks{
             Department of Mathematics, Graduate School of Science, 
             Osaka University. 
             1-1 Machikaneyama-cho, Toyonaka, Osaka 560-0043, Japan.
             (E-mail: {\tt sunagawa@math.sci.osaka-u.ac.jp})}
}

\date{\today }  

\maketitle

\noindent{\bf Abstract:}\ 
We consider the initial value problem for a three-component system of 
quadratic derivative nonlinear Schr\"odinger equations in two space dimensions 
with the masses satisfying the resonance relation. We present a structural 
condition on the nonlinearity under which small data global existence holds. 
It is also shown that the solution is asymptotically free. 
Our proof is based on the commuting vector field method combined with 
smoothing effects.
\\

\noindent{\bf Key Words:}\ 
Derivative nonlinear Schr\"odinger equation; 
Null condition; 
Mass resonance.\\

\noindent{\bf 2000 Mathematics Subject Classification:}\ 
35Q55, 35B40
\\

\section{Introduction}
We consider the initial value problem for the system of 
nonlinear Schr\"odinger equations
\begin{align}
 \left( i\pa_t+\frac{1}{2m_j}\Delta \right) u_j=F_j(u,\pa_x u), 
 \qquad  t>0,\ x \in \R^d,\ j=1,2,3
\label{Eq}
\end{align}
with
\begin{align}
 u_j(0,x)=\varphi_j(x), 
 \qquad  x\in \R^d,\ j=1,2,3,
\label{Data}
\end{align}
where $i=\sqrt{-1}$, 
each $m_j$ is a non-zero real constant,
$\pa_t=\pa/\pa t$, 
and
$\Delta=\sum_{a=1}^{d}\pa_{x_a}^2$
with $\pa_{x_a}=\pa/\pa x_a$
for $x=(x_a)_{a=1,\hdots,d} \in \R^d$.
$\varphi=(\varphi_j)_{j=1,2,3}$ is a prescribed $\C^3$-valued function, 
and 
$u=(u_j(t,x))_{j=1,2,3}$ is a $\C^3$-valued unknown function, 
while $\pa_x  u=(\pa_{x_a} u_j)_{a=1,\hdots,d; j=1,2,3}$ stands for its first order derivatives with respect to $x$. 
The nonlinear term $F=(F_j(u,\pa_x  u))_{j=1,2,3}$ is always assumed 
to be of the form 
\begin{align}
\left\{\begin{array}{l}
 \dis{F_1(u,\pa_x  u) =
  \sum_{|\alpha|,|\beta|\le 1} \Co{1}{\alpha}{\beta}  \, 
  (\overline{\pa^{\alpha} u_2})  (\pa^{\beta} u_3),
 }\\[7mm]
 \dis{F_2(u,\pa_x u) =
  \sum_{|\alpha|,|\beta|\le 1} \Co{2}{\alpha}{\beta}  \, 
   (\pa^{\alpha} u_3) (\overline{\pa^{\beta} u_1}),  
 }\\[7mm]
 \dis{F_3(u,\pa_x u) =
  \sum_{|\alpha|,|\beta|\le 1} \Co{3}{\alpha}{\beta}  \, 
  (\pa^{\alpha} u_1) (\pa^{\beta} u_2)
 }
\end{array}\right.
 \label{Nonlin}
\end{align}
with some complex constants $\Co{k}{\alpha}{\beta}$.

The system \eqref{Eq} appears in various physical settings 
(see, e.g., \cite{CC}, \cite{CCO}). 
If the derivatives are not included in $F$, 
this system reads
\begin{equation}
\begin{cases}
 \left( i\pa_t+\frac{1}{2m_1}\Delta \right) u_1= \overline{u_2} u_3, \\
 \left( i\pa_t+\frac{1}{2m_2}\Delta \right) u_2= u_3 \overline{u_1}, \\
 \left( i\pa_t+\frac{1}{2m_3}\Delta \right) u_3= u_1 u_2,
\end{cases}
 \quad  t>0,\ x \in \R^d.
\label{NLS3}
\end{equation}
Note also that the two-component system 
\begin{equation}
\begin{cases}
 \left( i\pa_t+\frac{1}{2m_1}\Delta \right) u_1= \overline{u_1} u_2, \\
 \left( i\pa_t+\frac{1}{2m_2}\Delta \right) u_2= u_1^2, 
\end{cases}
 \quad  t>0,\ x \in \R^d
\label{NLS2}
\end{equation}
can be regarded as a degenerate case of \eqref{NLS3}. 
In the case of $d=2$, Hayashi--Li--Naumkin \cite{HLN1} obtained a 
small data global existence result for \eqref{NLS2} under the 
relation 
\begin{align}
 m_2=2m_1.
 \label{Mass2}
\end{align}
The non-existence of the usual scattering state for \eqref{NLS2} is 
also proved under \eqref{Mass2}. 
Higher dimensional case ($d \ge 3$) for \eqref{NLS2} under the relation 
\eqref{Mass2} is considered by Hayashi--Li--Ozawa \cite{HLO} from 
the viewpoint of small data scattering.  
Remark that \eqref{Mass2} is often called the mass resonance 
relation, which was first discovered in the study of nonlinear Klein-Gordon 
systems (see \cite{S1}, \cite{T2}, \cite{DFX}, \cite{KS2}, 
\cite{KOS}, etc., and the references cited therein). 
The above-mentioned results for the two-component system \eqref{NLS2} 
can be generalized to the three-component system \eqref{NLS3} 
if the mass resonance relation \eqref{Mass2} is replaced by 
\begin{align}
 m_3=m_1+m_2.
 \label{Mass3}
\end{align}
However, it is non-trivial at all whether or not these can be 
generalized to the case of \eqref{Eq} under \eqref{Mass3}, 
because the presence of the derivatives in the nonlinearity causes a derivative loss in general. 
On the other hand,  the presence of the derivatives in the nonlinearity 
sometimes yields extra-decay property.  
One of the most successful example will be the null condition 
introduced by Christodoulou \cite{Chr} and Klainerman \cite{K} in the case of 
quadratic quasilinear systems of wave equations in three space dimensions. 
Our aim in this paper is to reveal analogous null structure 
in the case of quadratic derivative nonlinear Schr\"odinger systems. 
Of our particular interest is  the case of $d=2$, 
because the two-dimensional case for the Schr\"odinger equations 
corresponds to the three-dimensional case for wave equations  
from the viewpoint of the decay rate of solutions to the linearized 
equations.

In what follows, we concentrate our attention on the three-component 
Schr\"odinger system \eqref{Eq} with \eqref{Nonlin} 
under the relation \eqref{Mass3}. 
Remark that \eqref{Mass3} enables us to use the Leibniz-type rules 
for the operator $J_m(t)=x+\frac{it}{m}\pa_x$ (see Lemma~\ref{Lem_leibniz} below), 
which play a crucial role in our analysis. 
It should be noted that single quadratic nonlinear Schr\"odinger 
equations are distinguished from the present setting 
because \eqref{Mass3} is never satisfied when $m_1=m_2=m_3 \ne 0$. 
We refer to \cite{BHN}, \cite{De}, \cite{GMS}, \cite{HN} and the references 
cited therein for recent results on small data global existence and 
asymptotic behavior of solutions to single quadratic nonlinear 
Schr\"odinger equations in two space dimensions. 
We hope that our approach might be generalized to more general settings, 
but we do not pursue it here to avoid several technical complications.

\section{Main Results}

First of all, we should formulate the null condition for the Schr\"odinger 
case in an appropriate way. 
One way of understanding the null condition for quasi-linear wave equations
is the John-Shatah observation: ``The requirement that no
plane wave solution is genuinely nonlinear leads to the null condition"
(see \cite{J} for the detail; see also \cite{A} for the application to elastic waves).
Another way is to see the null condition as the condition to guarantee
the cancellation of the main part of the nonlinearity, which is 
sometimes called the H\"ormander test (\cite{hor}). 

Let us take the second way. For the solution $u_j^0$ to $\left(i\pa_t+\frac{1}{2m_j}\Delta\right)u_j^0=0$ with $u_j^0(0)=\varphi_j$, 
it is known that we have
$$
\pa_x^\alpha u_j^0(t,x)
\sim 
\left(\frac{im_jx}{t}\right)^\alpha 
\left({\frac{m_j}{it}}\right)^{\frac{d}{2}}
\widehat{\varphi_j}\left(\frac{m_jx}{t}\right) 
e^{\frac{im_j|x|^2}{2t}}
=: \left(\frac{im_jx}{t}\right)^\alpha
\psi_j(t,x) e^{\frac{im_j|x|^2}{2t}}
$$
as $t \to \infty$, 
where $\hat{\varphi}$ denotes the Fourier transform of $\varphi$. 
Then, under the mass resonance relation \eqref{Mass3}, 
the direct calculation yields 
$$
F_j({u^0,\pa_x u^0})\sim p_j\left(\frac{x}{t}\right)
\Psi_j(t,x) e^{\frac{im_j|x|^2}{2t}},
 \quad j=1,2,3,
$$
where 
\begin{align*}
 &p_1(\xi) =
 \sum_{|\alpha|,|\beta|\le 1} \Co{1}{\alpha}{\beta}  \, 
 (\overline{im_2 \xi})^{\alpha}\, (im_3\xi)^{\beta}, 
 \\
 &p_2(\xi) =
 \sum_{|\alpha|,|\beta|\le 1} \Co{2}{\alpha}{\beta}  \, 
 (im_3 \xi)^{\alpha} \, (\overline{im_1\xi})^{\beta},
 \\
 &p_3(\xi) =
 \sum_{|\alpha|,|\beta|\le 1} \Co{3}{\alpha}{\beta}  \, 
 (im_1\xi)^{\alpha} \,  (im_2\xi)^{\beta},
\end{align*}
with $\Psi_1=\overline{\psi_2}\psi_3$, $\Psi_2=\psi_3\overline{\psi_1}$,
$\Psi_3=\psi_1\psi_2$.
In order that these main parts of $F_j$ vanish,
we must have $p_1(\xi)=p_2(\xi)=p_3(\xi)=0$. 

The above observation naturally leads us to the following definition:

\begin{defn}\label{Def_null} 
We say that the nonlinear term $F=(F_j)_{j=1,2,3}$ of the form \eqref{Nonlin} satisfies 
the null condition if $p_1(\xi)=p_2(\xi)=p_3(\xi)=0$ for all $\xi \in \R^d$.
\end{defn}

\begin{rem}
In the case of one-dimensional cubic nonlinear Schr\"odinger equations, 
similar structural conditions have been considered in \cite{T1}, \cite{KT}, 
\cite{HN0}, \cite{KS1}, \cite{S2}, \cite{HNS}, etc. 
Analogous consideration for quadratic nonlinear 
Klein-Gordon systems can be found in \cite{DFX}, \cite{KS2}, \cite{KOS}. 
However, as far as the authors know, there are no previous papers 
which concern quadratic derivative nonlinear Schr\"odinger systems from 
this viewpoint. \\
\end{rem}

Now we are going to state our results. 
For a non-negative integer $s$, we denote by $H^{s}(\R^d)$ the standard 
Sobolev space: 
$$
 H^{s}(\R^d)= \bigl\{
 \phi \,;\, \pa_x^{\alpha}\phi(x) \in L^2(\R^d) 
  \ \mbox{for all $|\alpha|\le s$} \bigr\}
$$
equipped with the norm
$$
 \|\phi\|_{H^{s}(\R^d)}
 =
 \left(\sum_{|\alpha|\le s}\|\pa_x^{\alpha} \phi\|_{L^2(\R^d)}^{2}\right)^{1/2}.
$$
We also introduce the following function space: 
$$
 \Sigma^{s}(\R^d)= \bigl\{
 \phi \, ;\, x^{\alpha}\phi(x) \in H^{s-|\alpha|}(\R^d) 
  \ \mbox{for all $|\alpha|\le s$} \bigr\}
$$
equipped with the norm
$$
 \|\phi\|_{\Sigma^{s}(\R^d)}
 =
 \left(\sum_{|\alpha|\le s}\|x^{\alpha} \phi\|_{H^{s-|\alpha|}(\R^d)}^{2}\right)^{1/2}.
$$

The main result of this paper is the following.
\begin{thm} \label{Thm1}
Let $\varphi \in \Sigma^{s}(\R^2)$ with $s\ge 7$, 
and $F$ be of the form \eqref{Nonlin}. 
Assume that \eqref{Mass3} is satisfied and 
the nonlinear term $F$ satisfies the null condition in the sense of 
Definition~$\ref{Def_null}$. 
Then there exists a positive constant $\eps_1$ such that 
\eqref{Eq}--\eqref{Data} admits a unique global solution 
$u\in 
   C\bigl([0,\infty);\Sigma^{s}(\R^2)\bigr)$,
provided that $\|\varphi\|_{\Sigma^{s}(\R^2)} \le \eps_1$. 
Moreover, the solution $u(t)$ has a free profile, i.e., 
there exists $\varphi^+=(\varphi_j^+)_{j=1,2,3} \in \Sigma^{s-1}(\R^2)$ 
such that 
$$
 \lim_{t \to \infty}
 \sum_{j=1}^{3} 
 \|U_{m_j}(-t)u_j(t)- \varphi_j^{+}\|_{\Sigma^{s-1}(\R^2)}
 = 0,
$$
where $U_{m}(t)=\exp\left(\frac{it}{2m}\Delta\right)$. 
\end{thm}
\begin{rem} 
As we have mentioned in the introduction, small data global existence 
for \eqref{NLS3} has been proved although the nonlinearity in \eqref{NLS3} 
does not satisfy the  null condition in the sense of 
Definition~\ref{Def_null}. However, it is also shown in  \cite{HLN1} that 
the solution to \eqref{NLS3} does not have the free profile, which should be 
contrasted with Theorem~\ref{Thm1}. 
If one tries  to show the global existence in the case where the 
null condition is violated, some long-range effects must be taken into 
account 
{(see \cite{HLN1}, \cite{HLN2}, \cite{KLS}, etc., for the related works)}. 
\end{rem}

If we do not assume the null condition, our approach does not imply the small 
data global existence in two-dimensional case. 
However, we can prove the almost global existence. 

\begin{thm} \label{Thm2}
Let $\varphi \in \Sigma^{s}(\R^2)$ with $s\ge 7$, 
and $F$ be of the form \eqref{Nonlin}. 
Assume that \eqref{Mass3} is satisfied. 
Then there exist  positive constants $\eps_2$ and ${\omega}$ such that 
\eqref{Eq}--\eqref{Data} admits a unique solution 
$u \in 
 C\bigl([0,T];\Sigma^{s}(\R^2)\bigr)$
with some $T\ge \exp({\omega}/\eps)$, provided that 
$ \|\varphi\|_{\Sigma^{s}(\R^2)} =:\eps \le \eps_2$. 
\end{thm}

In the higher dimensional case ($d\ge 3$), we are able to show the following 
small data global existence result without assuming the null condition: 

\begin{thm} \label{Thm3}
Let $d\ge 3$, $\varphi \in \Sigma^{s}(\R^d)$ with 
$s\ge 2\left[\frac{d}{2}\right]+5$, 
and $F$ be of the form \eqref{Nonlin},
where $\left[\frac{d}{2}\right]$ is the largest integer not exceeding $\frac{d}{2}$. 
Assume that \eqref{Mass3} is satisfied. 
Then there exists a positive constant $\eps_3$ such that 
\eqref{Eq}--\eqref{Data} admits a unique global solution 
$u \in 
  C\bigl([0,\infty);\Sigma^{s}(\R^d)\bigr)$,
provided that $\|\varphi\|_{\Sigma^{s}(\R^d)} \le \eps_3$. 
Moreover, 
there exists $\varphi^+=(\varphi_j^+)_{j=1,2,3} \in \Sigma^{s-1}(\R^d)$ 
such that 
$$
 \lim_{t \to \infty}
 \sum_{j=1}^{3} 
 \|U_{m_j}(-t)u_j(t)- \varphi_j^{+}\|_{\Sigma^{s-1}(\R^d)}
 = 0.
$$
\end{thm}

\begin{rem} 
 In the case of $d=1$, 
 our approach does not imply any global nor almost global 
 existence results because of insufficiency of expected decay in $t$ of 
 the quadratic nonlinear term (see Remarks~\ref{1DRemark} 
 and \ref{Rmk_loc_1d} below  for the detail). 
 In the recent paper by Ozawa--Sunagawa \cite{OS}, 
 a small data blow-up result is obtained for a three-component system of 
 quadratic nonlinear Schr\"odinger equations in one space dimension. 
 More precisely, it is shown in \cite{OS} that we can  choose 
 $m_j$, $F_j$ and $\varphi_j$ with $\|\varphi\|_{\Sigma^s}=\eps$ 
 such that the corresponding solution to \eqref{Eq}--\eqref{Data} blows up 
 in finite time no matter how small $\eps>0$ is. 
 However, the  nonlinear term treated in \cite{OS} is different from 
 \eqref{Nonlin}.
 \\
\end{rem}

The rest of this paper is organized as follows. The next section is devoted to 
preliminaries on basic properties of the operator $J_m$. 
In Section~4, a characterization of the null condition will be given in terms 
of some special quadratic forms. As a consequence, extra-decay property of 
the nonlinear term satisfying the null condition under mass resonance will 
be made clear. 
In Section~5, we recall the smoothing property of linear Schr\"odinger 
equations. After that, the main theorems will be  proved in Section~6 
by means of {\it a priori} estimates. 
The Appendix is devoted to the proof of technical lemmas. 

In what follows, we denote several positive constants by the same 
letter $C$, which may vary from one line to another. 
Also we will frequently use the following convention on implicit constants: 
The expression $f=\sum_{\lambda \in \Lambda}' g_{\lambda}$ means that there 
exists a family $\{A_{\lambda}\}_{\lambda \in \Lambda}$ of constants such 
that $f=\sum_{\lambda \in \Lambda} A_{\lambda} g_{\lambda}$. 
For $z\in \R^d$ (or $z\in \R$), we write $\jb{z}=\sqrt{1+|z|^2}$.

\section{The operator $J_m$}
For a
non-zero real constant $m$, we set 
$L_m=i\pa_t+\frac{1}{2m}\Delta$ and 
$J_m(t)=\bigl(J_{m,a}(t)\bigr)_{a=1,\hdots,d}$ with $J_{m,a}(t)=x_a+\frac{it}{m}\pa_{x_a}$. 
Then we can easily check that 
\begin{align}  
 [L_m, \pa_{x_a}]=[L_m, J_{m,a}(t)]=0, \quad
 [J_{m,a}(t), \pa_{x_b}]=-\delta_{ab}, \quad 
 [J_{m,a}(t), J_{m,b}(t)]=0,
 \label{Comm}
\end{align}
where $[\cdot,\cdot]$ denotes the commutator of linear operators, and 
$\delta_{ab}$ is the Kronecker symbol, i.e.,
$$
 \delta_{ab}=\left\{
 \begin{array}{ll}
 1& (a=b),\\
 0& (a\ne b).
 \end{array}
 \right.
$$ 
We write 
$J_m(t)^{\alpha}=J_{m,1}(t)^{\alpha_{1}}\cdots J_{m,d}(t)^{\alpha_{d}}$ 
for a multi-index $\alpha=(\alpha_1,\hdots,\alpha_{d}) \in (\Z_+)^{d}$. 
Here and the later on as well, $\Z_+$ denotes the set of all non-negative 
integers. We also set
$$
 \Gamma_m(t)=(\pa_{x_1},\hdots,\pa_{x_d}, J_{m,1}(t),\hdots, J_{m,d}(t)),
$$
$$
 \Gamma_m(t)^{\alpha}
 =\pa_{x_1}^{\alpha_1}\cdots \pa_{x_d}^{\alpha_{d}}
  J_{m,1}(t)^{\alpha_{d+1}}\cdots J_{m,d}(t)^{\alpha_{2d}}
$$
for a multi-index $\alpha=(\alpha_1,\hdots,\alpha_{2d}) \in (\Z_+)^{2d}$. 
For simplicity of exposition, we often write $J_{m,a}$, $J_m$ and 
$\Gamma_{m}$ for $J_{m,a}(t)$, $J_m(t)$ and $\Gamma_{m}(t)$, 
respectively.

The following identities are quite useful: 
\begin{align}
 J_{m,a}(t) f 
  =
   \frac{it}{m} e^{im\theta} \pa_{x_a} (e^{-im\theta} f)
  =
   U_m(t) (x_a U_m(-t) f), 
   \label{Id_jm}
\end{align}
where 
$\theta=\frac{|x|^2}{2t}$
 and $U_m(t)=\exp(\frac{it}{2m}\Delta)$. 
Indeed, we can deduce the following lemmas from \eqref{Id_jm}.

\begin{lem}\label{Lem_kla}
We have
$$
 \|f\|_{L^{\infty}(\R^d)} \le \frac{C_m}{\jb{t}^{\frac{d}{2}} }
 \sum_{|\alpha| \le [\frac{d}{2}]+1}\|\Gamma_m(t)^{\alpha}f\|_{L^2(\R^d)}
$$
for $t\in \R$, 
where $C_m>0$ is independent of $t$. 
\end{lem}

\begin{proof}
We have only to consider the case of $t\ge 1$, because the opposite case 
easily follows from the standard Sobolev inequality. 
By the Gagliardo-Nirenberg-Sobolev inequality (see, e.g., \cite{Fr}), 
we have 
$$
 \|f\|_{L^{\infty}} 
 \le 
 C \|f\|_{L^2}^{1-\frac{d}{2\sigma}} 
 \biggl( 
  \sum_{|\alpha|= \sigma}\|\pa_x^{\alpha} f\|_{L^2}
 \biggr)^{\frac{d}{2\sigma}},
$$
where $\sigma= \left[\frac{d}{2}\right]+1$. 
Also we note that \eqref{Id_jm} yields
$$
 J_{m}^{\alpha}f
 = 
 \Bigl(\frac{it}{m}\Bigr)^{|\alpha|} 
 e^{im \theta} \pa_x^{\alpha} (e^{-im\theta}f)
$$
for any $\alpha \in (\Z_+)^d$. From them it follows that
\begin{align*}
  \|f\|_{L^{\infty}} 
  &= 
  \|e^{-im \theta} f\|_{L^{\infty}} 
  \\
  &\le 
  C \|e^{-im \theta}f\|_{L^2}^{1-\frac{d}{2\sigma}} 
 \biggl(
  \sum_{|\alpha|= \sigma}\|\pa_x^{\alpha} (e^{-im \theta} f)\|_{L^2}
 \biggr)^{\frac{d}{2\sigma}}
 \\
 &\le 
 C \|f\|_{L^2}^{1-\frac{d}{2\sigma}} 
 \biggl(
  |m|^\sigma t^{-\sigma}\sum_{|\alpha|= \sigma}\|J_{m}^{\alpha}f\|_{L^2}
 \biggr)^{\frac{d}{2\sigma}}
 \\
  &\le 
 C|m|^{\frac{d}{2}} t^{-\frac{d}{2}} \sum_{|\alpha|\le \sigma}\|J_{m}^{\alpha}f\|_{L^2}.
\end{align*}
This completes the proof.
\end{proof}

\begin{lem} \label{Lem_leibniz}
Let $m_1$, $m_2$, $m_3$ be 
non-zero real constants satisfying \eqref{Mass3}. 
Then we have 
\begin{align*} 
 &J_{m_1,a}(t)(\overline{f} g)
 =
  - \frac{m_2}{m_1}\bigl(\overline{ J_{m_2,a}(t) f} \bigr) g
  + \frac{m_3}{m_1} \overline{f} \bigl( J_{m_3,a}(t) g),\\
 &J_{m_2,a}(t)( f\overline{g})
 =
    \frac{m_3}{m_2} \bigl( J_{m_3,a}(t) f)\overline{g} 
  - \frac{m_1}{m_2} f\bigl(\overline{ J_{m_1,a}(t) g} \bigr),\\
 &J_{m_3,a}(t)(fg)
 =\frac{m_1}{m_3} \bigl( J_{m_1,a}(t) f \bigr) g 
  + \frac{m_2}{m_3}  f \bigl(J_{m_2,a}(t)g \bigr)   
\end{align*}
for $a\in \{1,\hdots,d\}$ and smooth functions $f$, $g$. 
\end{lem}
\begin{proof} 
Since $m_1=-m_2+m_3$, it follows from \eqref{Id_jm} that 
\begin{align*}
 J_{m_1,a}(\overline{f} g)
 &=
 \frac{it}{m_1} \overline{e^{i m_2 \theta}} e^{im_3 \theta}
  \pa_{x_a}\bigl\{ (\overline{e^{-im_2\theta}f}) (e^{-im_3\theta}g) \bigr\}\\
 &=
 \frac{it}{m_1} \Bigl\{
 \overline{e^{im_2\theta} \pa_{x_a} (e^{-im_2\theta}f)} g 
 +
 \overline{f} e^{im_3\theta} \pa_{x_a}(e^{-im_3\theta}g) \Bigr\}\\
 &=
 -\frac{m_2}{m_1} 
 \Bigl\{\overline{\frac{it}{m_2}e^{im_2\theta} \pa_{x_a} (e^{-im_2\theta}f)}
 \Bigr\} g 
 +\frac{m_3}{m_1} \, 
 \overline{f} \, 
 \Bigl\{\frac{it}{m_3}e^{im_3\theta} \pa_{x_a}(e^{-im_3\theta}g)\Bigr\} \\
 &=
 - \frac{m_2}{m_1}\bigl(\overline{ J_{m_2,a} f} \bigr) g
  + \frac{m_3}{m_1} \overline{f} \bigl( J_{m_3,a} g).
\end{align*}
The  other two identities can be shown in the same way. 
\end{proof}

\section{Characterization of the null condition}

Throughout this section, we fix 
non-zero real constants $m_1$, $m_2$, $m_3$
satisfying \eqref{Mass3}. 
Remember that our null condition depends on the masses. 
Let us introduce the following quadratic forms: 
\begin{align*} 
 &G_{1,a}(f,g)
   = m_2 \overline{f} (\pa_{x_a} g) + m_3 (\overline{\pa_{x_a} f}) g,\\
 &G_{2,a}(f,g)
   = m_3  f (\overline{\pa_{x_a} g}) + m_1 (\pa_{x_a} f) \overline{g},\\
 &G_{3,a}(f,g)= m_1 f (\pa_{x_a} g) - m_2  (\pa_{x_a} f) g
\end{align*}
and
$$
 Q_{ab}(f,g)=(\pa_{x_a} f)(\pa_{x_b} g)-(\pa_{x_b} f)(\pa_{x_a} g)
$$ 
for $a,b\in \{1,\hdots,d\}$. 
We call $G_{1,a}$, $G_{2,a}$, $G_{3,a}$  
{\it the null gauge forms} associated with the mass resonance relation 
\eqref{Mass3}, 
while $Q_{ab}$ is called {\it the strong null forms} 
(cf.~\cite{T1}, \cite{KT}, \cite{HK}, \cite{Geo90}). 

The objectives of this section are twofold: 
The first is to give an algebraic characterization of the  null condition 
in terms  of the null gauge forms and the strong null forms 
(Lemma~\ref{Lem_null_char}). 
The second is to investigate properties of these quadratic forms 
in connection with the operator $J_m$ (Lemmas~\ref{Lem_extra} and 
\ref{Lem_act_j}). 
As a consequence, extra-decay structure of $F$ satisfying 
the null condition under mass resonance will be made clear 
(Corollary~\ref{Cor_extra}).

We start with the following lemma. 

\begin{lem} \label{Lem_null_char}
The nonlinear term $F=(F_j)_{j=1,2,3}$ of the form \eqref{Nonlin} satisfies 
the null condition if and only if it can be written in the following form: 
\begin{align}
 &F_1(v,\pa_x v) =
  \mathop{{\;\,\sum}'}_{a=1}^{d} \, 
  G_{1,a}(v_2,v_3)
  +
  \mathop{{\;\,\sum}'}_{a,b=1}^{d} Q_{ab}  (\overline{v_2}, v_3),
 \label{Null_char}\\
 &F_2(v,\pa_x v) =
  \mathop{{\;\,\sum}'}_{a=1}^{d} \, 
  G_{2,a}(v_3,v_1)
  +
  \mathop{{\;\,\sum}'}_{a,b=1}^{d} Q_{ab}  (v_3, \overline{v_1}),
 \\
 &F_3(v,\pa_x v) =
  \mathop{{\;\,\sum}'}_{a=1}^{d} \, 
  G_{3,a}(v_1,v_2)
  +
  \mathop{{\;\,\sum}'}_{a,b=1}^{d} Q_{ab}  (v_1, v_2)
\end{align}
for any $\C^3$-valued smooth functions $v=(v_j)_{j=1,2,3}$. 
\end{lem}

\begin{rem}
An analogous characterization of the null condition for quadratic nonlinear 
Klein-Gordon systems can be found in Proposition~{5.1} of \cite{KOS}. 
See also \cite{Side}, \cite{A} for related works on nonlinear elastic 
wave equations. 
\end{rem}

\begin{proof}
We will consider only $F_1$ because the others can be shown in the 
same way. We set $o=(0,\hdots,0)\in (\Z_+)^d$ and  
$\iota(a)=(0,\hdots,0,\stackrel{\tiny\mbox{$a$-th}}{1},0,\hdots,0)
\in (\Z_+)^d$ for $a\in \{1,\hdots,d\}$. 
By the definition of $p_1(\xi)$, we have 
\begin{align*}
&p_1(\xi)
=
\Co{1}{o}{o} 
+i\sum_{a=1}^{d} (m_3 \Co{1}{o}{\iota(a)}-m_2 \Co{1}{\iota(a)}{o}) \xi_a\\
&
+ m_2m_3 \sum_{a=1}^{d} \Co{1}{\iota(a)}{\iota(a)} {\xi_a}^2
+ m_2m_3\sum_{1\le a<b\le d} 
 (\Co{1}{\iota(a)}{\iota(b)} + \Co{1}{\iota(b)}{\iota(a)})\xi_a \xi_b,
\end{align*}
where $\Co{1}{\alpha}{\beta}$ is the constant appearing in \eqref{Nonlin}.
In order that this polynomial vanish identically on $\R^d$, we must have 
\begin{align*}
 \begin{array}{ll}
 \Co{1}{o}{o} = \Co{1}{\iota(a)}{\iota(a)} =0
 & \mbox{for $a\in\{1,\hdots,d\}$},\\
 m_3 \Co{1}{o}{\iota(a)}-m_2 \Co{1}{\iota(a)}{o}=0  
 & \mbox{for $a\in\{1,\hdots,d\}$},\\
 \Co{1}{\iota(a)}{\iota(b)} + \Co{1}{\iota(b)}{\iota(a)}=0
 & \mbox{for $a, b\in\{1,\hdots,d\}$ with $a\ne b$}.
  \end{array}  
\end{align*}
These identities imply that $F_1$ is of the form \eqref{Null_char}. 
It is easy to see that the converse is also true.
\end{proof}

Next we turn our attention to properties of the null gauge forms 
and the strong null forms. The following lemma asserts a gain 
of extra-decay in $t$ with the aid of $J_m$.

\begin{lem} \label{Lem_extra}
{\rm (i)}\ \  For $a \in \{1,\hdots,d\}$, we have 
\begin{align*}
 G_{1,a}(f,g)=&
 \frac{im_2m_3}{t}\Bigl\{
 (\overline{J_{m_2, a} f}) g - \overline{f} (J_{m_3, a}g) 
 \Bigr\}, \\
 G_{2,a}(f,g)=&
 \frac{im_3m_1}{t}\Bigl\{
  -(J_{m_3, a}f)\overline{g} +f(\overline{J_{m_1, a} g}) 
 \Bigr\}, \\
 G_{3,a}(f,g)=&
 \frac{i m_1m_2}{t}\Bigl\{
 (J_{m_1, a} f) g -{f} (J_{m_2, a}g)
 \Bigr\}.
\end{align*}
{\rm (ii)}\ \ 
For $a,b \in \{1,\hdots,d\}$ and for 
$m$, $\mu$ $\in \R\setminus \{0\}$, 
we have 
\begin{align*}
 Q_{a b}(f,g)
=&
 \frac{\mu}{it}
 \Bigl\{(\pa_{x_a} f)(J_{\mu,b}g)-(\pa_{x_b} f)(J_{\mu,a}g)  \Bigr\}\\
 &\ \ +
 \frac{m}{it}\Bigl\{(J_{m,a} f)(\pa_{x_b} g)-(J_{m,b} f)(\pa_{x_a} g) \Bigr\}\\
 &\ \ +
 \frac{m\mu}{t^2}
 \Bigl\{(J_{m,a} f)(J_{\mu,b}g)-(J_{m,b} f)(J_{\mu,a}g) \Bigr\},\\
 Q_{a b}(\overline{f},g)
=&
 \frac{\mu}{it}
 \Bigl\{
  (\overline{\pa_{x_a} f})(J_{\mu,b}g)- (\overline{\pa_{x_b} f})(J_{\mu,a}g) 
 \Bigr\}\\
 &\ \ -
 \frac{m}{it}
 \Bigl\{
  (\overline{J_{m,a} f})(\pa_{x_b} g)-(\overline{J_{m,b} f})(\pa_{x_a} g) 
 \Bigr\}\\
 &\ \ -
 \frac{m\mu}{t^2}
 \Bigl\{(\overline{J_{m,a} f})(J_{\mu,b}g)-(\overline{J_{m,b} f})(J_{\mu,a}g) \Bigr\}.
\end{align*}

\end{lem}

\begin{proof}
{\rm (i)} follows immediately from the relation 
$\pa_{x_a}=\frac{im}{t}x_a-\frac{im}{t} J_{m,a}$. 
To prove {\rm (ii)}, we just substitute $mx_a=m J_{m,a}-it \pa_{x_a}$ and 
$\mu x_b=\mu  J_{\mu,b}-it \pa_{x_b}$ into 
the identities 
\begin{align*}
&(mx_a f)(\mu x_b g)-(mx_b f)(\mu x_a g) =0,\\
&(\overline{mx_a f})(\mu x_b g)-(\overline{mx_b f})(\mu x_a g) =0.
\end{align*}
This completes the proof.
\end{proof}

As for the action of the operator $J_m$ on the null gauge forms and 
the strong null forms, we have the following:

\begin{lem} \label{Lem_act_j}
Assume that the mass resonance relation \eqref{Mass3} is satisfied. 
Then we have 
\begin{align}
 &J_{m_1}^{\alpha} G_{1,a}(f,g)
 =
 \mathop{{\;\,\sum}'}_{|\beta|+|\gamma|\le |\alpha|} \, 
  G_{1,a}(J_{m_2}^{\beta}f, J_{m_3}^{\gamma}g),
 \\
 &J_{m_2}^{\alpha} G_{2,a}(f,g)
 =
 \mathop{{\;\,\sum}'}_{|\beta|+|\gamma|\le |\alpha|} \, 
  G_{2,a}(J_{m_3}^{\beta}f, J_{m_1}^{\gamma}g),
 \\
 &J_{m_3}^{\alpha} G_{3,a}(f,g)
 =
 \mathop{{\;\,\sum}'}_{|\beta|+|\gamma|\le |\alpha|} \, 
  G_{3,a}(J_{m_1}^{\beta}f, J_{m_2}^{\gamma}g),
 \label{J_g}
\end{align}
and
\begin{align}
 J_{m_1}^{\alpha} Q_{ab}(\overline{f},g)
 =&
 \mathop{{\;\,\sum}'}_{|\beta|+|\gamma|\le |\alpha|} \, 
  Q_{ab}(\overline{J_{m_2}^{\beta}f}, J_{m_3}^{\gamma}g) 
 \nonumber\\
 &+
 \sum_{c=1}^{d}\mathop{{\;\,\sum}'}_{|\beta|+|\gamma|\le |\alpha|-1} \, 
  G_{1,c}(J_{m_2}^{\beta}f, J_{m_3}^{\gamma}g),
 \\
 J_{m_2}^{\alpha} Q_{ab}(f, \overline{g})
 =&
 \mathop{{\;\,\sum}'}_{|\beta|+|\gamma|\le |\alpha|} \, 
  Q_{ab}(J_{m_3}^{\beta}f, \overline{J_{m_1}^{\gamma}g}) 
 \nonumber\\
 &+
 \sum_{c=1}^{d}\mathop{{\;\,\sum}'}_{|\beta|+|\gamma|\le |\alpha|-1} \, 
  G_{2,c}(J_{m_3}^{\beta}f, J_{m_1}^{\gamma}g),
 \\
 J_{m_3}^{\alpha} Q_{ab}(f,g)
 =&
 \mathop{{\;\,\sum}'}_{|\beta|+|\gamma|\le |\alpha|} \, 
  Q_{ab}(J_{m_1}^{\beta}f, J_{m_2}^{\gamma}g) 
 \nonumber\\
 &+
 \sum_{c=1}^{d}\mathop{{\;\,\sum}'}_{|\beta|+|\gamma|\le |\alpha|-1} \, 
  G_{3,c}(J_{m_1}^{\beta}f, J_{m_2}^{\gamma}g)
 \label{J_q}
\end{align}
for any $\alpha \in (\Z_+)^{d}$ and $a,b \in \{1,\hdots,d\}$.
\end{lem}

\begin{proof}
We will show only \eqref{J_g} and \eqref{J_q}, because the others can 
be shown in the same way. 
From  Lemma~\ref{Lem_leibniz} and the commutation relation \eqref{Comm}, 
it follows that
\begin{align*}
 J_{m_3,b}G_{3,a}(f,g)
 =&
 m_1\left\{
  \frac{m_1}{m_3} (J_{m_1,b} f) (\pa_{x_a} g)
  + \frac{m_2}{m_3} f ( J_{m_2,b} \pa_{x_a} g)
 \right\}\\
 &-
 m_2\left\{
  \frac{m_1}{m_3}(J_{m_1,b} \pa_{x_a} f ) g
  + \frac{m_2}{m_3} (\pa_{x_a} f) ( J_{m_2,b} g)
 \right\}\\
 =&
 \frac{m_1}{m_3}\Bigl\{
 m_1 (J_{m_1,b} f) (\pa_{x_a} g)
 -
 m_2 (J_{m_1,b} \pa_{x_a} f) g
 \Bigr\}\\
 &+
 \frac{m_2}{m_3} \Bigl\{
 m_1 f ( J_{m_2,b} \pa_{x_a} g)
 -
 m_2 (\pa_{x_a} f) ( J_{m_2,b} g)
 \Bigr\}\\
 =&
 \frac{m_1}{m_3}\Bigl\{
 G_{3,a}(J_{m_1,b} f, g)
 +
 m_2 \delta_{ba}  f g
 \Bigr\}\\
 &+
 \frac{m_2}{m_3} \Bigl\{
 G_{3,a} (f, J_{m_2,b}g)
 -
 m_1 f \delta_{ba} g
 \Bigr\}\\
 =&
  \frac{m_1}{m_3} G_{3,a}(J_{m_1,b} f, g)
 +
 \frac{m_2}{m_3}  G_{3,a} (f, J_{m_2,b}g).
\end{align*}
Similarly we have 
\begin{align*}
J_{m_3,c}Q_{ab}(f,g) =&
 \left\{
  \frac{m_1}{m_3} (J_{m_1,c} \pa_{x_a} f) (\pa_{x_b} g)
  +
  \frac{m_2}{m_3} (\pa_{x_a} f) (J_{m_2,c} \pa_{x_b} g)
 \right\}\\
 &\ \  - 
 \left\{
  \frac{m_1}{m_3} (J_{m_1,c} \pa_{x_b} f) (\pa_{x_a} g)
  +
  \frac{m_2}{m_3}(\pa_{x_b} f) (J_{m_2,c} \pa_{x_a} g)
 \right\}
\\
=&
 \frac{m_1}{m_3}\Bigl\{
 (J_{m_1,c} \pa_{x_a} f) (\pa_{x_b} g)-(J_{m_1,c} \pa_{x_b} f) (\pa_{x_a} g)
  \Bigr\}\\ 
 &\ \  + 
 \frac{m_2}{m_3} \Bigl\{
 (\pa_{x_a} f) (J_{m_2,c} \pa_{x_b} g)-(\pa_{x_b} f) (J_{m_2,c} \pa_{x_a} g)
  \Bigr\} \\
=&
 \frac{m_1}{m_3}\Bigl\{Q_{ab}(J_{m_1,c}f,g)
 -\delta_{ca}f(\pa_{x_b} g) + \delta_{cb}f(\pa_{x_a} g) 
  \Bigr\}\\ 
 &\ \  + 
 \frac{m_2}{m_3} \Bigl\{Q_{ab}(f,J_{m_2,c} g)
 -(\pa_{x_a} f) \delta_{cb} g + (\pa_{x_b} f) \delta_{ca} g
  \Bigr\} \\
 =&
 \frac{m_1}{m_3}Q_{ab}(J_{m_1,c}f,g)
 +
  \frac{m_2}{m_3} Q_{ab}(f,J_{m_2,c} g)
 +
 \frac{\delta_{cb}}{m_3} G_{3,a}(f,g)\\
 &\ \ -
 \frac{\delta_{ca}}{m_3} G_{3,b}(f,g).
\end{align*}
These equalities imply \eqref{J_g} and \eqref{J_q} with $|\alpha|=1$. 
By induction on $\alpha$, we arrive at the desired conclusion.
\end{proof}

\begin{cor} \label{Cor_extra}
Assume that the masses satisfy \eqref{Mass3} 
and that the nonlinear term $F=(F_j)_{j=1,2,3}$ of the form \eqref{Nonlin} satisfies the null condition. 
Then we have
$$
 \sum_{j=1}^{3}\sum_{|\alpha|\le s} |\Gamma_{m_j}^{\alpha}F_j(v,\pa_x v)|
 \le 
 \frac{C}{\jb{t}} 
 \left(
  \sum_{j=1}^{3}\sum_{|\alpha|\le \left[\frac{s}{2}\right]+1} 
  |\Gamma_{m_j}^{\alpha}v_j|
 \right)
 \left( 
  \sum_{j=1}^{3}\sum_{|\alpha|\le s+1} |\Gamma_{m_j}^{\alpha}v_j| 
 \right)
$$
for any $s \in \Z_+$ and $v=(v_j)_{j=1,2,3}$, 
where the positive constant $C$ is independent of $t$.
\end{cor}

\begin{proof}
By Lemma~\ref{Lem_null_char}, Lemma~\ref{Lem_act_j} and 
the usual Leibniz rule, we have 
\begin{align*}
 &\Gamma_{m_1}^{\alpha}F_1(v,\pa_x v) \\
  &\ \ =
  \sum_{|\beta| + |\gamma| \le |\alpha|} \Bigl\{
  \mathop{{\;\,\sum}'}_{a=1}^{d} \, 
  G_{1,a}(\Gamma_{m_2}^{\beta}v_2,\Gamma_{m_3}^{\gamma}v_3)
  +
  \mathop{{\;\,\sum}'}_{a,b=1}^{d} 
  Q_{ab} (\overline{\Gamma_{m_2}^{\beta}v_2},\Gamma_{m_3}^{\gamma}v_3)
 \Bigr\},
 \\
 &\Gamma_{m_2}^{\alpha}F_2(v,\pa_x v) \\
 &\ \ =
  \sum_{|\beta| + |\gamma| \le |\alpha|} \Bigl\{
 \mathop{{\;\,\sum}'}_{a=1}^{d} \, 
  G_{2,a}(\Gamma_{m_3}^{\beta}v_3,\Gamma_{m_1}^{\gamma}v_1)
  +
  \mathop{{\;\,\sum}'}_{a,b=1}^{d} 
  Q_{ab}  (\Gamma_{m_3}^{\beta}v_3, \overline{\Gamma_{m_1}^{\gamma} v_1})
 \Bigr\},
 \\
 &\Gamma_{m_3}^{\alpha}F_3(v,\pa_x v) \\
 &\ \ =
  \sum_{|\beta| + |\gamma| \le |\alpha|} \Bigl\{
  \mathop{{\;\,\sum}'}_{a=1}^{d} \, 
  G_{3,a}(\Gamma_{m_1}^{\beta}v_1, \Gamma_{m_2}^{\gamma}v_2)
  +
  \mathop{{\;\,\sum}'}_{a,b=1}^{d} 
  Q_{ab}  (\Gamma_{m_1}^{\beta}v_1, \Gamma_{m_2}^{\gamma}v_2)
 \Bigr\}
\end{align*}
for any $\alpha \in (\Z_+)^{2d}$. 
To obtain the desired estimate, we have only to apply Lemma \ref{Lem_extra} 
for each terms on the right-hand sides of the above identities. 
\end{proof}

\section{Smoothing effect}
In this section, we recall smoothing properties of the linear Scr\"odinger 
equations. As we have mentioned in the introduction, the system \eqref{Eq} 
does not allow the standard energy estimate because the  nonlinearity contains 
the derivatives of the unknown. Smoothing effect is a useful device 
to overcome this difficulty. 
Among various versions  of smoothing  effects, 
we will mainly follow the approach of \cite{BHN} with a few modifications 
to fit our purpose (see also \cite{KPV}, \cite{Do}, \cite{Chi4}, etc., 
for the history and more information on this subject).

Let ${\mathcal S}(\R^d)$ be the space of rapidly decreasing functions,
and the Fourier transform of $f\in {\mathcal S}(\R^d)$ be given by
$$
\widehat{f}(\xi)\bigl(={\mathcal F}[f](\xi)\bigr)
=\frac{1}{(2\pi)^{d/2}}\int_{\R^d}e^{-i{y}\cdot\xi} f({y}) d{y}.
$$
For $a \in \{1,\hdots,d\}$,  we denote by $\mathcal{H}_a$ 
the Hilbert transform with respect to $x_a$, that is, 
$$
 (\mathcal{H}_a f)(x) 
 =
 \frac{1}{\pi}\mathrm{p.v.} \int_{\R} f(x-\tau \mathbf{1}_a)
 \frac{d\tau}{\tau}
$$
for $f\in {\mathcal S}(\R^d)$,
where $\mathbf{1}_a=(\delta_{a b})_{b=1,\hdots,d} \in \R^d$.
${\mathcal H}_a$ can be uniquely extended to a bounded linear 
operator on $L^2(\R^d)$ satisfying 
$\widehat{{\mathcal H}_a f}(\xi)=-i\sgn{\xi_a}\widehat{f}(\xi)$. 
With a parameter $\kappa \in (0,1]$, we put 
$$
{\Lambda}_{\kappa, a}(t,x)=\kappa \arctan
\left(\frac{x_a}{\jb{t}}\right)
$$ 
and 
$$
 S_{\pm, a}(t; \kappa) f
 = 
 \bigl(\cosh {\Lambda}_{\kappa, a}(t,\cdot) \bigr)f
 \mp 
 i \bigl(\sinh {\Lambda}_{\kappa, a}(t,\cdot) \bigr) \mathcal{H}_a f
$$
for $f\in L^2(\R^d)$ and $t\in \R$.
We define the  operators $S_{\pm}(t; \kappa)$ by 
$$
 S_{\pm}(t; \kappa)=\prod_{a=1}^{d} S_{\pm, a}(t; \kappa)
$$
for $t\in \R$.
Since 
$\|\bigl(\tanh {\Lambda}_{\kappa, a}(t,\cdot) \bigr)
{\mathcal H}_a\|_{L^2\to L^2}
\le \tanh(\pi/2)<1$, 
both $S_{\pm}(t; \kappa)$ and their inverse operators $S_{\pm}(t; \kappa)^{-1}$, which are given by
$$
S_{\pm}(t; \kappa)^{-1} f
=
\prod_{a=1}^d 
 \left(
  I\mp i \tanh {\Lambda}_{\kappa, a}(t,\cdot){\,\mathcal H}_a
 \right)^{-1}
\left(\frac{f}{\cosh {\Lambda}_{\kappa, a}(t,\cdot)}\right),
$$
are bounded operators on $L^2(\R^d)$ with 
the estimates 
\begin{equation}
\label{EstS}
\sup_{t\in \R, \kappa\in (0,1]}\|S_{\pm}(t; \kappa)\|_{L^2\to L^2}<\infty, \quad  
\sup_{t\in \R, \kappa\in (0,1]}\|S_{\pm}(t; \kappa)^{-1}\|_{L^2\to L^2}<\infty.
\end{equation}

Roughly speaking, if we put $S(t)=S_+(t; \kappa)$ when $m>0$ and 
$S(t)=S_-(t; \kappa)$ when $m<0$, 
then the operator $S(t)$ is expected to satisfy 
$$
 [L_m, S(t)] \simeq 
 -\frac{i\kappa}{|m|\jb{t}} 
 \sum_{a=1}^{d} w_a^2(t,\cdot)  S(t)|\pa_{x_a}| +\mbox{remainder},
$$  
where $|\pa_{x_a}|$ is interpreted as the Fourier multiplier and 
$$
 w_a(t,x)=\jbf{\frac{x_a}{\jb{t}}}^{-1}=\biggl(1+\frac{x_a^{2}}{1+t^2}\biggr)^{-\frac{1}{2}}.
$$
This enables us to recover a half-derivative of the solution 
(Lemma~\ref{Lem_smoothing}). 
By combining with auxiliary estimates (Lemma~\ref{Lem_aux}), 
we can get rid of the worst contribution of the nonlinear term 
(Lemma~\ref{Lem_QLE}). 

\begin{lem} \label{Lem_smoothing}
Let  
$m \in \R\setminus \{0\}$, $\kappa \in (0,1]$, $t_0\in\R$, and $T>0$. 
Put 
$S(t)=S_+(t;\kappa)$ when $m>0$ and $S(t)=S_-(t;\kappa)$ when $m<0$.
We have 
\begin{align*}
& \|S(t)f(t)\|_{L^2}^2 
 +
 \int_{t_0}^t \frac{\kappa}{|m|\jb{\tau}}
 \sum_{a=1}^{d} 
\Bigl\| w_a(\tau) S(\tau)|\pa_{x_a}|^{\frac{1}{2}}f(\tau) \Bigr\|_{L^2}^2 d\tau\\
&\le 
 \|S(t_0)f(t_0)\|_{L^2}^2+\int_{t_0}^t\biggl( 
 \frac{C\kappa}{\jb{\tau}} \|S(\tau)f(\tau)\|_{L^2}^2
 +
 2\left|\Jb{S(\tau)f(\tau), S(\tau) L_mf(\tau)}_{L^2}\right| \biggr)d\tau
\end{align*}
for $t\in [t_0, t_0+T]$ and $f\in C\bigl([t_0, t_0+T]; H^2(\R^d)\bigr)\cap C^1\bigl([t_0, t_0+T]; L^2(\R^d)\bigr)$,
where the constant $C$ is independent of $\kappa \in (0,1]$, $T>0$ and $t_0\in \R$. 
\end{lem}

\begin{lem}  \label{Lem_aux}
%
%
%
%
%
Let $m \in \R\setminus \{0\}$ and $\kappa\in (0,1]$. 
Let $S$, $S'$ be either $S_+(t;\kappa)$ or $S_-(t;\kappa)$.
We have 
\begin{align*}
 &\left| \Jb{S f, S(g \pa_{x_a} h)}_{L^2}   \right|
  +
  \left| \Jb{S f, S (g \overline{\pa_{x_a} h})}_{L^2} \right|\\
&\le
  \frac{C}{\jb{t}^{\frac{d}{2}}}
 \left(
  \sum_{|\alpha|\le [\frac{d}{2}]+3}\|\Gamma_m^{\alpha} g \|_{L^2} 
 \right) \left( 
   \|f\|_{L^2} 
   +
   \left\| w_a(t) S|\pa_{x_a}|^{\frac{1}{2}}f \right\|_{L^2}
 \right) \\
 &\qquad \times 
 \left(  
   \|h\|_{L^2} 
   +
   \left\| w_a(t) S'|\pa_{x_a}|^{\frac{1}{2}}h \right\|_{L^2}
 \right),
\end{align*}
where the constant $C$ is independent of $\kappa \in (0,1]$ and $t \in \R$. 
\end{lem}

We will give a sketch of the proof of these lemmas in the Appendix. \\

Let $m_j\in \R\setminus\{0\}$ for $j\in\{1,2,3\}$, $t_0\in \R$ and $T>0$.
We consider the system
\begin{equation}
\label{QLS}
L_{m_j}v_j
=
\sum_{k=1}^3 \sum_{a=1}^d 
 \left(
  {\Phi}_{jk, a} \pa_{x_a}v_k + {\Psi}_{jk,a}\pa_{x_a} \overline{v_k}
 \right)
 +
 {\Theta}_j,
\quad (t,x)\in (t_0,t_0+T)\times \R^d
\end{equation}
for $j=1,2,3$ with given functions $\Phi_{jk,a}$, $\Psi_{jk,a}$ and $\Theta_j$. 
For $s\in \Z_+$ and an interval $I$, {we denote by} 
${\mathcal B}^s(I\times \R^d)$ the space of functions of $C^s$-class on 
$I\times\R^d$ with bounded derivatives of order up to $s$. 

\begin{lem}\label{Lem_QLE} 
Let $d\ge 2$, and let ${\lambda}_{jk}, \mu_{jk}\in \R\setminus\{0\}$ 
be given. Suppose that 
$$
 {\Phi}_{jk,a}, {\Psi}_{jk,a}
  \in {\mathcal B}^1([t_0, t_0+T]\times \R^d)
      \cap 
      C\left([t_0,t_0+T];\Sigma^{\left[\frac{d}{2}\right]+3}(\R^d)\right),
$$
and
${\Theta}=({\Theta}_j)_{j=1,2,3}\in 
L^1\bigl(t_0,t_0+T; L^2(\R^d)\bigr)$. 
Let $v\in C\bigl([t_0,t_0+T]; L^2(\R^d)\bigr)$ satisfy \eqref{QLS}.
There is a positive constant $\delta$ such that if
\begin{align*}
 &e_{t_0,T} :=\\
 &\sup_{t\in [t_0,t_0+T]} 
  \sum_{j,k=1}^3\sum_{a=1}^d\sum_{|\alpha|\le [\frac{d}{2}]+3}
  \left(
      \|\Gamma_{\lambda_{jk}}(t)^\alpha {\Phi}_{jk,a}(t)\|_{L^2}
      +
      \|\Gamma_{\mu_{jk}}(t)^\alpha {\Psi}_{jk,a}(t)\|_{L^2}
 \right)\\
 &\le \delta,
\end{align*}
then we have
\begin{align*}
 \|v(t)\|_{L^2}^2
 \le & 
 C \|v(t_0)\|_{L^2}^2
 +
 {C} \int_{t_0}^t \left(
    \frac{e_{t_0,T}}{\jb{\tau}}\|v(\tau)\|_{L^2}^2
    +
    \|v(\tau)\|_{L^2}\|{\Theta}(\tau)\|_{L^2}
  \right) d\tau 
\end{align*}
for $t\in [t_0, t_0+T]$. Here the positive constants $\delta$ and $C$ depend 
only on $|m_j|$, $|{\lambda}_{jk}|$ and $|\mu_{jk}|$. 
In particular, they are independent of $t_0$ and $T$.
\end{lem}

\begin{proof}
Here we prove the estimate only for the case where 
$${\Theta}\in C\bigl([t_0,t_0+T]; L^2(\R^d)\bigr)$$ 
and 
$$v \in C\bigl([t_0,t_0+T]; H^2(\R^d)\bigr)
      \cap C^1\bigl([t_0,t_0+T]; L^2(\R^d)\bigr).$$ 
The general case follows from the standard argument of mollifiers (in time and 
space variables), where we use the assumption that the coefficients belong to 
${\mathcal B}^1\bigl([t_0,t_0+T]\times \R^d\bigr)$ (see \cite{Miz} for 
instance).

We may assume $e_{t_0,T}>0$, because the standard energy estimate gives the 
desired result if $e_{t_0,T}=0$. 
For each $j\in\{1,2,3\}$, we put $S_j(t)=S_+(t;\kappa)$ if $m_j>0$, and 
$S_j(t)=S_-(t;\kappa)$ if $m_j<0$, where the parameter $\kappa\in (0,1]$ 
will be fixed later. By Lemma~\ref{Lem_aux}, we get
\begin{align*}
2&\sum_{j=1}^{3}\left|\jb{S_j(\tau)v_j(\tau), 
       S_j(\tau)L_{m_j}v_j(\tau)}_{L^2}\right|\\
&\le  
{C_*} \sum_{j=1}^3\sum_{a=1}^d
  \frac{e_{t_0,T}}{|m_j|\jb{\tau}^{\frac{d}{2}}} 
 \left(
  \left\|w_aS_{j}(\tau)|\pa_{x_a}|^{\frac{1}{2}}v_{j}(\tau)\right\|_{L^2}^2
  + 
  \|v_{j}(\tau)\|_{L^2}^2
 \right)\\
&\quad{}+
  2\sum_{j=1}^{3}
   \|S_j(\tau)v_j(\tau)\|_{L^2} \|S_j(\tau)\Theta_j(\tau)\|_{L^2}
\end{align*}
with some positive constant $C_*$ depending only on $|m_j|$, $|\lambda_{jk}|$ 
and 
$|\mu_{jk}|$ (in particular, $C_*$ is independent of $\kappa$).
Note that we have $\|v_k(t)\|_{L^2}\le C\|S_k(t)v_k(t)\|_{L^2}$ 
by \eqref{EstS}. 
We put 
$$
\kappa:=C_*e_{t_0,T},\quad \delta=\frac{1}{C_*}.
$$
If $0<e_{t_0,T}\le \delta$, we have $0<\kappa\le 1$ and 
it follows from Lemma~\ref{Lem_smoothing} that
\begin{align}
&\sum_{j=1}^3\|S_j(t)v_j(t)\|_{L^2}^2
 \nonumber\\
&\le  
-\sum_{j=1}^3
  \int_{t_0}^t
   \frac{\kappa}{|m_j|\jb{\tau}}\sum_{a=1}^d
   \left\|w_aS_j(\tau)|\pa_{x_a}|^{\frac{1}{2}}v_j(\tau)\right\|_{L^2}^2 
  d\tau
 +
 \sum_{j=1}^3\|S_j(t_0) v_j(t_0)\|_{L^2}^2
 \nonumber\\
&\quad {}+\sum_{j=1}^3
  \int_{t_0}^t
   \left(
     \frac{C\kappa}{\jb{\tau}}\|S_j(\tau)v_j(\tau)\|_{L^2}^2
     +
     2\left|
      \jb{S_j(\tau)v_j(\tau), S_j(\tau)L_{m_j}v_j(\tau)}_{L^2}
     \right|
   \right)
  d\tau
\nonumber\\
&\le  
  \sum_{j=1}^3\int_{t_0}^t
   \frac{C_*e_{t_0,T}-\kappa}{|m_j|\jb{\tau}}
   \sum_{a=1}^d
     \left\|
       w_aS_j(\tau)|\pa_{x_a}|^{\frac{1}{2}}v_j(\tau)
     \right\|_{L^2}^2 
  d\tau
  +
  \sum_{j=1}^3\|S_j(t_0)v_j(t_0)\|_{L^2}^2
\nonumber\\
&\quad {}+
\sum_{j=1}^3\int_{t_0}^t
  \left\{ 
  \frac{Ce_{t_0,T}}{\jb{\tau}}
  \bigl( \|S_j(\tau)v_j(\tau)\|_{L^2}^2 + \|v_j(\tau)\|_{L^2}^2 \bigr)
 +
  2\|S_j(\tau)v_j(\tau)\|_{L^2}\|S_j(\tau)\Theta_j(\tau)\|_{L^2}
  \right\} 
 d\tau
 \nonumber\\
 &\le 
 \ \sum_{j=1}^3\|S_j(t_0)v_j(t_0)\|_{L^2}^2
 +
 C \int_{t_0}^t
  \left(\frac{e_{t_0,T}}{\jb{\tau}}\|v(\tau)\|_{L^2}^2
  +
   \|v(\tau)\|_{L^2}\|\Theta(\tau)\|_{L^2}\right)
 d\tau
\label{Est_S_v}
\end{align}
for $t\in [t_0,t_0+T]$, since $d\ge 2$.
Recalling \eqref{EstS} again, we obtain the desired result.
\end{proof}

\begin{rem}\label{1DRemark}
When $d=1$, we need to assume that 
$\left(\sup_{t\in[t_0,t_0+T]}\jb{t}\right)^{1/2}\, e_{t_0,T}$ 
is sufficiently small 
in order that the above proof works, because 
$\jb{\tau}^{-1/2} \ge \jb{\tau}^{-1}$.
\end{rem}

\section{Proof of the theorems}

Throughout this section, we always suppose that the mass resonance relation 
\eqref{Mass3}, as well as \eqref{Nonlin}, is satisfied. 
For $v:\R^d \ni x \mapsto (v_j(x))_{j=1,2,3}\in \C^3$ and $s \in \Z_+$, 
we set 
$$
 |v(x)|_{\Gamma(t),s}
 =
 \left(
  \sum_{j=1}^{3} \sum_{|\alpha|\le s} |\Gamma_{m_j}(t)^{\alpha} v_j(x)|^2
 \right)^{1/2}
$$
and
$$
 \|v\|_{\Gamma(t),s}=\bigl\| |v(\cdot)|_{\Gamma(t), s} \bigr\|_{L^2(\R^d)}.
$$
Note that $\|v\|_{\Gamma(t),s}$ and 
$\sum_{j=1}^{3}\|U_{m_j}(-t) v_j\|_{\Sigma^s}$ are equivalent as a norm, where 
$U_m(t)=\exp\left(\frac{it}{2m}\Delta\right)$ for $m\in \R\setminus\{0\}$. 
Indeed, \eqref{Id_jm} implies 
$$
\pa^{\alpha}J_{m_j}(t)^{\beta}v_j=U_{m_j}(t)\pa^{\alpha}\bigl( x^{\beta}
U_{m_j}(-t)v_j\bigr), \quad \alpha,\beta\in (\Z_+)^d. 
$$
Moreover, for any bounded interval $I(\subset\R)$ and $s\in \Z_+$, there is 
a positive constant $C_{I,s}$ such that 
$$
\frac{1}{C_{I, s}} \|v\|_{\Gamma(t),s}
\le 
 \|v\|_{\Sigma^s}
\le 
 C_{I,s} \|v\|_{\Gamma(t),s},
\quad t\in I.
$$
In fact, given $m\in \R\setminus\{0\}$, there are polynomials 
$p^{\alpha,m}_{\beta,\gamma}$ and $q^{\alpha,m}_{\beta,\gamma}$
of $t\in \R$ such that we have
$$
\Gamma_m(t)^\alpha f(x)=\sum_{|\beta|+|\gamma|\le |\alpha|}
p_{\beta,\gamma}^{\alpha, m}(t)\pa_x^\beta\bigl(x^\gamma f(x)\bigr),
\quad \alpha\in (\Z_+)^{2d},
$$
and
$$
\pa_x^\beta\bigl(x^\gamma f(x)\bigr)
=\sum_{|\alpha| \le |\beta|+|\gamma|}
q_{\beta,\gamma}^{\alpha, m}(t)\Gamma_m(t)^\alpha f(x),
\quad \beta,\gamma\in (\Z_+)^d
$$
for any sufficiently smooth function $f$.
In the sequel, we often write $|u(t, x)|_{\Gamma, s}$ and 
$\|u(t)\|_{\Gamma,s}$ for $|u(t, x)|_{\Gamma(t),s}$ and 
$\|u(t,\cdot)\|_{\Gamma(t),s}$, respectively.
We also write 
$u^{(\alpha)}=\bigl(u_{j}^{(\alpha)}\bigr)_{j=1,2,3}$, 
$F^{(\alpha)}=\bigl(F_{j}^{(\alpha)}\bigr)_{j=1,2,3}$ 
with 
$u_j^{(\alpha)}= \Gamma_{m_j}^{\alpha}u_j$, 
$F_j^{(\alpha)}=\Gamma_{m_j}^{\alpha}\bigl(F_j(u,\pa_x u)\bigr)$ 
for the simplicity of the exposition.
 
First we state the local existence result for \eqref{Eq} 
(considered for $t>t_0$)
with the initial condition 
\begin{align}
 u_j(t_0,x)=\psi_j(x),
 \quad  x\in \R^d,\ j=1,2,3,
\label{Data2}
\end{align}
with some $t_0\in \R$, instead of \eqref{Data}.
For the convenience of the readers, we will give an outline of the proof 
in the Appendix. 

\begin{lem}\label{Lem_local}
Let $d\ge 2$, $t_0\in \R$, and $B\ge 0$. 
For any $s\ge 2\left[\frac{d}{2}\right]+4$, there is a positive constant 
$\delta_s$, which is independent of $t_0$ and $B$, such that
for any $\psi=(\psi_j)_{j=1,2,3}\in \Sigma^s(\R^d)$ with
$$
\|\psi\|_{\Gamma(t_0),\left[\frac{d}{2}\right]+4}\le \eps <\delta_s
$$
and $\|\psi\|_{\Gamma(t_0),2\left[\frac{d}{2}\right]+4}\le B$,
the initial value problem \eqref{Eq}--\eqref{Data2} possesses a unique 
solution $u\in C\bigl([t_0, t_0+T_s^*]; \Sigma^s(\R^d)\bigr)$ with 
\begin{equation}
\label{DS3}
 \sup_{t \in[t_0, t_0+T_s^*]}
 \|u(t)\|_{\Gamma, \left[\frac{d}{2}\right]+4}<\delta_s,
\end{equation}
where $T_s^*=T_s^*(\eps, B)$ is a positive constant which can be determined 
only by $s$, $\eps$ and $B$, and is independent of $t_0$.
\end{lem}

\begin{rem} \label{Rmk_loc_1d}
In the case of $d=1$, 
this claim is true except that $T_s^*$ may depend also on $t_0$ 
(see Remark~\ref{1DRemark} as well as the proof of 
Lemma~\ref{LE_approx} in the Appendix). 
It is unclear whether this exception is just a technical one or not, 
but this problem is out of the purpose of the present work.
\end{rem}

By virtue of this lemma, 
it suffices to obtain an {\it a priori} estimate for 
$\|u(t)\|_{\Gamma, 2\left[\frac{d}{2}\right]+4}$
with a sufficiently small upper bound
in order to show the global 
or almost global existence for \eqref{Eq}--\eqref{Data}.
Let $u$ be a solution to \eqref{Eq}--\eqref{Data} 
satisfying $u\in C\bigl([0,T^*];\Sigma^{s}(\R^d)\bigr)$
with some $T^*>0$ and some positive integer 
$s(\ge 2\left[\frac{d}{2}\right]+4)$.

We define 
$$
 E_k(T)= \sup_{t \in [0,T]}\|u(t)\|_{\Gamma, k}
$$
for $T \in [0,T^*]$ and $k \in \Z_+$. 
Then we have the following:

\begin{lem} \label{Lem_est_low}
Let $d\ge 1$, $s \ge 2\left[\frac{d}{2}\right]+4$ and $T \in [0,T^*]$.\\
{\rm{(i)}}\ \ We have 
\begin{align}  
 E_{s-1}(T)\le 
 C_1 \|\varphi\|_{\Sigma^{s-1}}
 + 
 C_2E_{s-1}(T) 
 \int_{0}^{T}\frac{\|u(t)\|_{\Gamma,s}}{(1+t)^{\frac{d}{2}}}dt,
 \label{Est_low}
\end{align}
where the positive constants $C_1$, $C_2$ are independent of $T$. \\
{\rm (ii)}\ \ 
If the null condition is satisfied, then we have 
\begin{align}  
 E_{s-1}(T)\le 
 C_1 \|\varphi\|_{\Sigma^{s-1}}
 + 
 C_3E_{s-1}(T) 
 \int_{0}^{T}\frac{\|u(t)\|_{\Gamma,s}}{(1+t)^{1+\frac{d}{2}}}dt,
 \label{Est_low_null}
\end{align}
where the positive constant $C_3$ is independent of $T$. 
\end{lem}

\begin{proof}
By Lemmas~\ref{Lem_kla} and \ref{Lem_leibniz}, we have 
\begin{align*} 
  \sum_{|\alpha|\le s-1} \|F^{(\alpha)}(t)\|_{L^2}
   &\le 
  C
   \Bigl\| |u(t)|_{\Gamma, [\frac{s+1}{2}]} \Bigr\|_{L^\infty} 
   \|u(t)\|_{\Gamma,s}\\
   &\le 
  \frac{C}{\jb{t}^{\frac{d}{2}}} 
  \|u(t)\|_{\Gamma, s-1}\|u(t)\|_{\Gamma,s}
    \le 
  CE_{s-1}(T) \frac{\|u(t)\|_{\Gamma,s}}{\jb{t}^{\frac{d}{2}}}
\end{align*}
for $t\in [0, T]$.
Here we have used the relation 
$\left[\frac{s+1}{2}\right]+\left[\frac{d}{2}\right]+1 \le s-1$ 
for $s\ge 2\left[\frac{d}{2}\right]+4$.  
Also it follows from the commutation relation \eqref{Comm} that 
$L_{m_j}u_j^{(\alpha)} =F_j^{(\alpha)}$. 
Therefore the standard energy inequality leads to  
\begin{align*}
 \sum_{|\alpha|\le s-1}\|u^{(\alpha)}(t)\|_{L^2}
 &\le 
  \sum_{|\alpha|\le s-1}\|u^{(\alpha)}(0)\|_{L^2} 
 + 
 \int_0^t
  \sum_{|\alpha|\le s-1}\|F^{(\alpha)}(\tau) \|_{L^2}
 d\tau\\
 &\le 
 C\|\varphi\|_{\Sigma^{s-1}}
 +  
 C E_{s-1}(T)  
 \int_0^{T} \frac{\|u(\tau)\|_{\Gamma,s}}{\jb{\tau}^{\frac{d}{2}}} d\tau
\end{align*}
for $t\le T$. This implies \eqref{Est_low}. 
If the null condition is satisfied, we can use Corollary~\ref{Cor_extra}, 
instead of Lemma~\ref{Lem_leibniz}, to obtain 
\begin{align} 
  \sum_{|\alpha|\le s-1} \|F^{(\alpha)}(t)\|_{L^2}
   &\le 
  \frac{C}{\jb{t}} 
   \Bigl\|
     |u(t)|_{\Gamma, [\frac{s+1}{2}]}
   \Bigr\|_{L^\infty} 
     \|u(t)\|_{\Gamma,s}
   \nonumber\\
   &\le 
   \frac{C}{\jb{t}^{1+\frac{d}{2}}} 
   \|u(t)\|_{\Gamma, s-1}\|u(t)\|_{\Gamma,s}
   \nonumber\\
  &\le 
  CE_{s-1}(T)  \frac{\|u(t)\|_{\Gamma,s}}{\jb{t}^{1+\frac{d}{2}}}
  \label{DoubleStar}
\end{align}
for $t\in [0,T]$,
which, together with the energy inequality, yields \eqref{Est_low_null}.
\end{proof}

\begin{lem} \label{Lem_est_high}
Let $d\ge 2$, $s \ge 2\left[\frac{d}{2}\right]+5$ and $T \in [0,T^*]$. 
We put $s_0:=\left[\frac{d}{2}\right]+4$. 
There is a positive constant $\widetilde{\delta}$ such that if 
$E_{s_0}(T)\le \widetilde{\delta}$, we have 
\begin{align}  
 \|u(t)\|_{\Gamma, s}
 \le C_4\|\varphi\|_{\Sigma^{s}} (1+t)^{C_5E_{s-1}(T)}
 \label{Est_high}
\end{align}
for $t \in [0,T]$, 
where the positive constants $C_4$, $C_5$ and $\widetilde{\delta}$ are 
independent of $T$. 
\end{lem}

\begin{proof}
For $\alpha\in (\Z_+)^{2d}$, we write $\alpha=(\alpha',\alpha'')$
with $\alpha', \alpha''\in (\Z_+)^d$. 
Let $|\alpha|\le s$. By using Lemma \ref{Lem_leibniz}, we can split 
$F_j^{(\alpha)}$ into the following form: 
\begin{align*}
 &F_1^{(\alpha)}
 =
 \sum_{a=1}^{d}
 \Bigl( 
 g_{1, a}^{\alpha} \pa_{x_a} u_3^{(\alpha)} 
 +
 h_{1,a}^{\alpha} \overline{\pa_{x_a} u_2^{(\alpha)}}
 \Bigr)
 + 
 r_1^{\alpha},\\
 &F_2^{(\alpha)}
 =
 \sum_{a=1}^{d}
 \Bigl( 
 g_{2, a}^{\alpha} \overline{\pa_{x_a} u_1^{(\alpha)}} 
 +
 h_{2,a}^{\alpha} \pa_{x_a} u_3^{(\alpha)}
 \Bigr)
 + 
 r_2^{\alpha},
 \\
 &F_3^{(\alpha)}
 =
 \sum_{a=1}^{d}
 \Bigl( 
 g_{3, a}^{\alpha} \pa_{x_a} u_2^{(\alpha)}
 +
 h_{3,a}^{\alpha} \pa_{x_a} u_1^{(\alpha)} 
 \Bigr)
 + 
 r_3^{\alpha},
\end{align*}
where 
\begin{align*}
 \begin{array}{l}
 \displaystyle
 g_{1, a}^{\alpha}= 
  \left(\frac{m_3}{m_1} \right)^{|\alpha''|}
  \sum_{|\beta|\le 1} \Co{1}{\beta}{\iota(a)}  (\overline{\pa^{\beta} u_2}), 
 \\
 \displaystyle
 h_{1, a}^{\alpha}= 
  \left(\frac{-m_2}{m_1}\right)^{|\alpha''|} 
 \sum_{|\beta|\le 1} \Co{1}{\iota(a)}{\beta}  (\pa^{\beta} u_3),
 \\
  \displaystyle
 g_{2, a}^{\alpha}= 
  \left(\frac{-m_1}{m_2} \right)^{|\alpha''|}
  \sum_{|\beta|\le 1} \Co{2}{\beta}{\iota(a)}  (\pa^{\beta} u_3), 
 \\
 \displaystyle
 h_{2, a}^{\alpha}= 
  \left(\frac{m_3}{m_2}\right)^{|\alpha''|} 
 \sum_{|\beta|\le 1} \Co{2}{\iota(a)}{\beta}  (\overline{\pa^{\beta} u_1}),
 \\
  \displaystyle
 g_{3, a}^{\alpha}= 
  \left(\frac{m_2}{m_3} \right)^{|\alpha''|}
  \sum_{|\beta|\le 1} \Co{3}{\beta}{\iota(a)} (\pa^{\beta} u_1), 
 \\
 \displaystyle
 h_{3, a}^{\alpha}= 
  \left(\frac{m_1}{m_3}\right)^{|\alpha''|} 
 \sum_{|\beta|\le 1} \Co{3}{\iota(a)}{\beta}  (\pa^{\beta} u_2),
\end{array}
\end{align*}
and $r_j^{\alpha}$ satisfies
$$
 |r_j^{\alpha}| \le C|u|_{\Gamma,[\frac{s}{2}]+1} |u|_{\Gamma,s}.
$$

From \eqref{Eq} we get $u\in C^1\bigl([0,T]; \Sigma^{s-2}(\R^d)\bigr)$,
which, together with the Sobolev embedding theorem, implies 
$g_{j,a}^\alpha, h_{j,a}^\alpha\in {\mathcal B}^1([0,T]\times \R^d)$,
since we have $\left[\frac{d}{2}\right]+2\le s-2$.
We set 
$(\lambda_1, \mu_1)=(-m_2, m_3)$, 
$(\lambda_2, \mu_2)=(m_3, -m_1)$ 
and 
$(\lambda_3, \mu_3)=(m_1, m_2)$. 
Then we get
\begin{align}
 \sum_{|\alpha|\le s}\sum_{j=1}^3\sum_{a=1}^d\sum_{|\gamma|\le [\frac{d}{2}]+3} 
 \Bigl(\|\Gamma_{\lambda_j}^{\gamma} g_{j, a}^{\alpha}(t)\|_{L^2}
 {}+\|\Gamma_{\mu_j}^{\gamma} h_{j, a}^\alpha(t)\|_{L^2}\Bigr)
  &\le 
  C\|u(t)\|_{\Gamma,[\frac{d}{2}]+4}
 \nonumber\\
  &\le CE_{s_0}(T)
 \label{Est_coeff1}
\end{align}
for $0\le t\le T$.
We also have
$$
 \|r_j^{\alpha}(t)\|_{L^2} 
  \le 
  \frac{C}{\jb{t}^{\frac{d}{2}}}
  \|u(t)\|_{\Gamma,[\frac{s}{2}] + [\frac{d}{2}] + 2}  
  \|u(t)\|_{\Gamma,s}
  \le
  CE_{s-1}(T) \frac{\|u(t)\|_{\Gamma,s}}{\jb{t}^{\frac{d}{2}}}
$$
for $0\le t\le T$.
Here we have used the relation 
$\left[\frac{s}{2}\right]+\left[\frac{d}{2}\right]+2\le s-1$ for 
$s\ge 2\left[\frac{d}{2}\right]+5$. 
In view of \eqref{Est_coeff1},  
we can apply Lemma~\ref{Lem_QLE} if $E_{s_0}(T)$ is sufficiently small,
and we obtain
 \begin{align*}
  \|u(t)\|_{\Gamma,s}^2
  \le &
  C \|\varphi\|_{\Sigma^s}^2
  +
  C \int_0^t\left(
    \frac{E_{s_0}(T)\|u(\tau)\|_{\Gamma,s}^2}{\jb{\tau}}
    +
    \sum_{j=1}^3\|r_j^\alpha(\tau)\|_{L^2}\|u(\tau)\|_{\Gamma,s}
  \right)d\tau\\
  \le & 
   C \|\varphi\|_{\Sigma^s}^2
   +
   CE_{s-1}(T)\int_{0}^{t}\frac{\|u(\tau)\|_{\Gamma,s}^2}{\jb{\tau}}d\tau
\end{align*}
for $t\in [0,T]$. 
By the Gronwall lemma, we obtain the desired estimate. 
\end{proof}

Now we are in a position to complete the proof of the main theorems. 

\begin{proof}[\underline{Proof of Theorem~$\ref{Thm1}$}] 
First we show the global existence for \eqref{Eq}--\eqref{Data}. 
Let $d=2$ and 
$s\ge 7=\left. \left(2\left[\frac{d}{2}\right]+5\right)\right|_{d=2}$. 
We put $A=4C_1$ and 
$\eps_0=\min\left\{
         \frac{\widetilde{\delta}}{A}, \frac{1}{2C_5 A}, \frac{1}{8C_3 C_4}
        \right\}$, 
where the constants $\widetilde{\delta}$ and $C_j$ with $1\le j\le 5$ come 
from the previous lemmas. We also define 
$$
 T^{**}=\sup \{ T \in [0,T^*] \  ;\  E_{s-1}(T)\le A\eps \}
$$
for  $\eps=\|\varphi\|_{\Sigma^{s}} \le \eps_0$
(note that we have $T^{**}>0$ because $E_{s-1}(0)\le C_1\eps\le A\eps/4$). 
Then it follows from \eqref{Est_high} and \eqref{Est_low_null} that 
$T^{**}=T^*$. 
Indeed, if $T^{**}<T^*$, we have $E_{s-1}(T^{**}) \le A\eps$ and
thus $E_{s_0}(T^{**})\le A\eps_0\le \widetilde{\delta}$. Hence 
\eqref{Est_high} implies 
$\|u(t)\|_{\Gamma,s} \le C_4 \eps (1+t)^{\frac{1}{2}}$, 
and \eqref{Est_low_null} yields 
\begin{align*}
 E_{s-1}(T^{**})
 \le 
 \frac{A}{4}\eps 
 + 
 C_3 A\eps \int_0^{\infty} \frac{C_4 \eps}{(1+t)^{\frac{3}{2}}}dt
 \le 
 \frac{A}{2}\eps.
\end{align*}
By the continuity of $[0, T^{*}]\ni T \mapsto E_{s-1}(T)$, we can choose 
$\tilde{T}>T^{**}$ such that $E_{s-1}(\tilde{T})\le A\eps$, which contradicts 
the definition of $T^{**}$, and the desired identity is shown. 
Consequently, we see that 
$$
  \sup_{t\in [0,T^*]} \|u(t)\|_{\Gamma,6}
  \le E_{s-1}(T^*) \le A \eps.
$$  
This {\it a priori} bound and Lemma~\ref{Lem_local} imply the global 
existence. Moreover, the global solution $u(t)$ satisfies
\begin{align} 
 \|u(t)\|_{\Gamma,s-1}\le C\eps,
 \quad 
 \|u(t)\|_{\Gamma,s} \le C\eps \jb{t}^{C\eps}
 \label{Est_sol}
\end{align} 
for all  $t\ge 0$, where the second inequality follows from the first one and 
\eqref{Est_high}. 

Next we turn our attention to the asymptotic behavior. 
We write $F_j(t)=F_j(u(t),\pa_x u(t))$ for simplicity.  
We observe that Corollary~\ref{Cor_extra} and \eqref{Est_sol} yield
\begin{align*}
 \sum_{j=1}^{3} \|U_{m_j}(-t) F_j(t)\|_{\Sigma^{s-1}}
 \le
 C\| F(t) \|_{\Gamma,s-1}
 \le 
 \frac{C \eps^2}{\jb{t}^{2-C\eps}}
 \in L^1(0,\infty)
\end{align*}
if $\eps$ is small enough (cf.~\eqref{DoubleStar}). This allows us to define 
$$
 \varphi_j^+:=  \varphi_j -
 i\int_{0}^{\infty} U_{m_j}(-\tau) F_j(\tau) d\tau
 \in \Sigma^{s-1}(\R^2).
$$
Then it follows from the Duhamel formula that 
\begin{align*}
 u_j(t)
 &=  U_{m_j}(t)\varphi_j -
  i\int_{0}^{t} U_{m_j}(t-\tau) F_j(\tau) d\tau\\
 &=U_{m_j}(t)\left(
   \varphi_j^+ 
   + i\int_{t}^{\infty} U_{m_j}(-\tau) F_j(\tau) d\tau
  \right).
\end{align*}
Therefore we have
\begin{align*} 
\|U_{m_j}(-t) u_j(t) - \varphi_j^+ \|_{\Sigma^{s-1}}
 &\le 
 \int_{t}^{\infty} \|U_{m_j}(-\tau) F_j(\tau)\|_{\Sigma^{s-1}} 
 d\tau\\
 &\le 
 \int_{t}^{\infty} \frac{C \eps^2}{\jb{\tau}^{2-C\eps}}
 d\tau\\ 
 &\le 
 \frac{C \eps^2}{\jb{t}^{1-C\eps}}. 
\end{align*}
This completes the proof.
\end{proof}

\begin{proof}[\underline{Proof of Theorem~$\ref{Thm2}$}] 
We put $A=4C_1$ and choose ${\omega}>0$ such that 
$$C_2C_4 e^{C_5 A {\omega}} {\omega} \le 1/4.$$ 
Then, it follows from \eqref{Est_high} and \eqref{Est_low} with $d=2$ 
that the inequalities 
$E_{s-1}(T)\le A\eps$ and $\log (1+T) \le {\omega}/\eps$ imply
\begin{align*}
 E_{s-1}(T)
 &\le 
 C_1\eps 
 + 
 C_2 A\eps \int_0^{T} 
 \frac{C_4 \eps e^{C_5 A\eps \log (1+T)}}{(1+t)}dt
 \\
 &\le 
 \frac{A}{4}\eps + C_2 C_4 e^{C_5 A{\omega}} A\eps^2 \log (1+T)
 \\
 &\le 
 \left(\frac{1}{4} + C_2 C_4 e^{C_5 A{\omega}} {\omega} \right) A\eps 
 \le 
 \frac{A}{2}\eps,
\end{align*}
provided that $\eps$ is small enough.
This means that we can keep $E_{s-1}(T^{*})$ dominated by $A\eps$ 
as long as  $\log (1+T^{*}) \le {\omega}/\eps$. 
This {\it a priori} bound and Lemma~\ref{Lem_local} 
imply the desired almost global existence.
\end{proof}

\begin{proof}[\underline{Proof of Theorem~$\ref{Thm3}$}] 
We skip it since the essential idea is exactly the same as 
that of the preceding ones. We only point out that 
$\jb{t}^{-\frac{d}{2}+C\eps} \in L^1(0,\infty)$ 
if $d\ge 3$ and $\eps$ is small enough.
\end{proof}

\appendix 
\section{Proof of Lemmas~\ref{Lem_smoothing}, \ref{Lem_aux} and \ref{Lem_local}}

This Appendix is devoted to the proof of Lemmas~\ref{Lem_smoothing}, 
\ref{Lem_aux} and \ref{Lem_local},
which are essentially not new but less trivial.

\subsection{Proof of Lemma~\ref{Lem_smoothing}}

The following lemma, which will be used repeatedly, 
is a special case of Lemma~2.1 in \cite{BHN} and we refer 
the readers to it for the proof. 

\begin{lem}\label{BeHaNa}
We have
$$
 \left\| \bigl[ |\pa_{x_a}|^{\frac{1}{2}}, g \bigr] f \right\|_{L^2}
 {}+\left\|\bigl[ |\pa_{x_a}|^{\frac{1}{2}}\mathcal{H}_a, g \bigr] f  
 \right\|_{L^2}\le
 C \|g\|_{L^{\infty}}^{\frac{1}{2}}
   \|\pa_{x_a}g\|_{L^{\infty}}^{\frac{1}{2}} 
   \|f\|_{L^2}
$$
for 
$f\in L^2(\R^d)$, $g\in W^{1,\infty}(\R^d)$ and $a\in \{1,\ldots,d\}$.
\end{lem}

We put $L_{m,\nu}=L_m-i\nu\Delta$ for $m\in \R\setminus\{0\}$ and $\nu\ge 0$.
For the later purpose, we show a slightly generalized version of 
Lemma~\ref{Lem_smoothing}.

\begin{lem} \label{Lem_smoothing2}
Let  
$m \in \R\setminus \{0\}$, $\kappa\in (0,1]$, $t_0 \in \R$, $T>0$, 
and $\nu \in [0,1]$. Put 
$S(t)=S_+(t;\kappa)$ when $m>0$ and $S(t)=S_-(t;\kappa)$ when $m<0$. 
We have 
\begin{align*}
& \|S(t)f(t)\|_{L^2}^2 
 +
 \int_{t_0}^t \frac{\kappa}{|m|\jb{\tau}}
 \sum_{a=1}^{d} 
\Bigl\| w_a(\tau) S(\tau)|\pa_{x_a}|^{\frac{1}{2}}f(\tau) \Bigr\|_{L^2}^2 d\tau
\\
& \le 
\|S(t_0)f(t_0)\|_{L^2}^2+\int_{t_0}^t\biggl( 
 \frac{C\kappa}{\jb{\tau}} \|S(\tau)f(\tau)\|_{L^2}^2
 +
 2\left|\Jb{S(\tau)f(\tau), S(\tau) L_{m,\nu}f(\tau)}_{L^2}\right| \biggr)d\tau
\end{align*}
for $t\in [t_0,t_0+T]$ and $f\in C\bigl([t_0, t_0+T]; H^2(\R^d)\bigr) \cap C^1\bigl([t_0,t_0+T]; L^2(\R^d)\bigr)$,
where the constant $C$ is independent of $\kappa\in (0,1]$, $\nu \in [0,1]$, $T>0$ and $t_0\in \R$. 
\end{lem}

\begin{proof}[\underline{Proof of Lemma~\ref{Lem_smoothing2}}]
As in the usual energy method, we first compute 
\begin{align*}
 \frac{1}{2}\frac{d}{dt}\|Sf\|_{L^2}^2
 =&
  \imagpart\Jb{L_{m, \nu} Sf,Sf}_{L^2}-\nu\|\nabla Sf\|_{L^2}^2\\
 \le&
 \imagpart \Jb{SL_{m, \nu} f,Sf}_{L^2} 
 +
  \imagpart \Jb{[L_{m, \nu}, S]f, Sf}_{L^2}.
\end{align*}
Also the straightforward calculation yields 
\begin{align*}
 [L_{m, \nu},S]
 =
 -\biggl(\frac{i}{|m|}+2\sgn{m}\nu\biggr)
  \sum_{a=1}^{d}(\pa_{x_a}{\Lambda}_{\kappa, a})S|\pa_{x_a}|
 +\sum_{a=1}^{d}R_a,
\end{align*}
where
\begin{align*}
 R_a 
 =& 
 \sgn{m} (\pa_t {\Lambda}_{\kappa, a}) S\mathcal{H}_a \\
 &+
  \biggl(\frac{1}{2m}-i\nu\biggr)
  \Bigl\{
   (\pa_{x_a} {\Lambda}_{\kappa, a})^2S  
    - 
   i\sgn{m}(\pa_{x_a}^2 {\Lambda}_{\kappa, a})S \mathcal{H}_a
  \Bigr\}.
\end{align*}
Remark that $|\pa_{x_a}|=\mathcal{H}_a\pa_{x_a}=\pa_{x_a}{\mathcal H}_a$ and 
$\mathcal{H}_a^2=-1$. Since 
$$
  |\pa_{t} {\Lambda}_{\kappa, a}| 
  + 
  |\pa_{x_a} {\Lambda}_{\kappa, a}|^2 
  + 
  |\pa_{x_a}^2 {\Lambda}_{\kappa, a}| 
  \le  
  \frac{C\kappa}{\jb{t}}+ \frac{C\kappa^2}{\jb{t}^2}
  \le  
  \frac{C\kappa}{\jb{t}}
$$
and $|-i\nu+\frac{1}{2m}|\le 1+\frac{1}{2|m|}$,
we have 
$$
 \|R_a f\|_{L^2} 
 \le 
 \frac{C\kappa}{\jb{t}}\|f\|_{L^2}
 \le 
 \frac{C\kappa}{\jb{t}}\|Sf\|_{L^2}.
$$
We also note that 
$$
 \pa_{x_a} {\Lambda}_{\kappa, a}=\frac{\kappa}{\jb{t}}w_a^2.
$$
Summing up, we obtain
\begin{align*}
 & \frac{d}{dt}\|Sf\|_{L^2}^2
 +
 \frac{2\kappa}{|m| \jb{t}}\sum_{a=1}^{d}
  \realpart \jbf{w_a S|\pa_{x_a}|f,w_a Sf }_{L^2}
\\ 
&\le 
 2\left|\jbf{SL_{m, \nu} f,Sf}_{L^2} \right|
 +
 \frac{C\kappa}{\jb{t}} \| Sf\|_{L^2}^2
 +\frac{4\kappa}{\jb{t}}\sum_{a=1}^{d}
  \bigl|\imagpart \jbf{w_a S|\pa_{x_a}|f,w_a Sf }_{L^2}\bigr|.
\end{align*}
Since we have
\begin{align*}
 &w_aS|\pa_{x_a}|f\\
 &= \pa_{x_a}\left(w_a S {\mathcal H}_af\right)
 +
 \left[w_aS, \pa_{x_a}\right]{\mathcal H}_af\\
 &= 
-|\pa_{x_a}|^{\frac{1}{2}}
 \left(
  \left[|\pa_{x_a}|^{\frac{1}{2}}{\mathcal H}_a, w_a S\right]{\mathcal H}_a f
  -
  w_aS|\pa_{x_a}|^{\frac{1}{2}}f
 \right)
+
\left[w_aS, \pa_{x_a}\right]{\mathcal H}_af,
\end{align*}
we get
\begin{align*}
& \bigl\| w_a S|\pa_{x_a}|^{\frac{1}{2}} f \bigr\|_{L^2}^2
-
 \JB{w_a S|\pa_{x_a}|f, w_a Sf}_{L^2}\\
& =
\jbf{
  w_aS|\pa_{x_a}|^{\frac{1}{2}}f, \left[w_aS,|\pa_{x_a}|^{\frac{1}{2}}\right]f
 }_{{L^2}}
+
 \jbf{
 \left[|\pa_{x_a}|^{\frac{1}{2}}{\mathcal H}_a,  w_a S\right]{\mathcal H}_a f,
 w_aS|\pa_{x_a}|^{\frac{1}{2}}f
 }_{{L^2}}\\
& \quad 
 +
 \jbf{
  \left[|\pa_{x_a}|^{\frac{1}{2}}{\mathcal H}_a, w_a S\right]{\mathcal H}_a f,
  \left[|\pa_{x_a}|^{\frac{1}{2}},w_aS\right]f
 }_{{L^2}}
 -
 \jbf{\left[w_aS,\pa_{x_a}\right]{\mathcal H}_a f, w_a S f}_{{L^2}}.
\end{align*}
Using Lemma~\ref{BeHaNa} with $g=w_a\cosh {\Lambda}_{\kappa,a}$, 
$w_a\sinh {\Lambda}_{\kappa,a}$, etc., 
we can show that all the commutators above are bounded operators on $L^2$. 
Hence we obtain
\begin{align*}
\bigl\| w_a S|\pa_{x_a}|^{\frac{1}{2}} f \bigr\|_{L^2}^2
 &-
 \realpart \JB{w_a S|\pa_{x_a}|f, w_a Sf}_{L^2}\\
&\le 
 C\bigl\| w_a S|\pa_{x_a}|^{\frac{1}{2}} f \bigr\|_{L^2} \|f\|_{L^2}
 +
 C\|f\|_{L^2}^2\\
&\le 
 \frac{1}{4}\bigl\| w_a S|\pa_{x_a}|^{\frac{1}{2}} f \bigr\|_{L^2}^2
 +
 C\|Sf\|_{L^2}^2,
\end{align*}
and 
\begin{align*}
\left|\imagpart \jbf{w_a S|\pa_{x_a}|f, w_a Sf}_{L^2}\right|
&
\le 
 C\bigl\| w_a S|\pa_{x_a}|^{\frac{1}{2}} f \bigr\|_{L^2} \|f\|_{L^2}
 +
 C\|f\|_{L^2}^2\\
&\le 
 \frac{1}{8|m|}\bigl\| w_a S|\pa_{x_a}|^{\frac{1}{2}} f \bigr\|_{L^2}^2
 +
 C\|Sf\|_{L^2}^2.
\end{align*}
Therefore we get
\begin{align*}
& \frac{d}{dt}\|Sf\|_{L^2}^2+\frac{\kappa}{|m|\jb{t}}\sum_{a=1}^d 
 \left\|w_aS|\pa_{x_a}|^{\frac{1}{2}}f\right\|_{L^2}^2
\le 2\left|\jb{S L_{m,\nu}f, S f}_{L^2}\right|+\frac{C\kappa}{\jb{t}}\|Sf\|_{L^2}^2.
\end{align*}
Finally, integrating this inequality, we obtain the desired result.
\end{proof}
 

\subsection{Proof of Lemma~\ref{Lem_aux}}

     %
     %
We follow the similar line as the proof of Lemma~{2.3} in \cite{BHN}. 
     %
Denoting by ${S}^*$ the adjoint operator of $S$, we have
\begin{align*}
&\jbf{
 S f, S(g\pa_{x_a}h)}_{L^2}\\
 &= \jbf{\overline{g}S^* S f, 
 (\pa_{x_a}w_a^{-1})w_a h}_{L^2}
{}
 - 
 \jbf{
  |\pa_{x_a}|^{\frac{1}{2}}\left(w_a^{-1}\overline{g}S^*Sf \right), 
  {\mathcal H}_a|\pa_{x_a}|^{\frac{1}{2}}\left(w_a h \right)
 }_{L^2}\\
&=:I_1+I_2,
\end{align*}
where we have used the identity 
\begin{align*}
\pa_{x_a}h
&=
(\pa_{x_a}w_a^{-1})w_a h + w_a^{-1}\pa_{x_a}(w_a h)\\
&=
(\pa_{x_a}w_a^{-1})w_a h 
- 
w_a^{-1}|\pa_{x_a}|^{\frac{1}{2}}
{\mathcal H}_a|\pa_{x_a}|^{\frac{1}{2}}(w_a h).
\end{align*}
It is easy to see 
$|I_1|\le C\|g\|_{L^\infty}\|f\|_{L^2}\|h\|_{L^2}$.
Since $\mathcal H_a (S')^{-1}$ is a bounded operator on $L^2$ and 
${\mathcal H}_a=(\mathcal H_a(S')^{-1})S'$, we get
\begin{align*}
\left\|
 {\mathcal H}_a |\pa_{x_a}|^{\frac{1}{2}}\left(w_a h \right)
\right\|_{L^2}
\le& 
C\left\|S'|\pa_{x_a}|^{\frac{1}{2}}\left(w_a h \right)\right\|_{L^2}
\\
\le & 
C\left(
\left\|w_a S'|\pa_{x_a}|^{\frac{1}{2}}h \right\|_{L^2}
+
\left\|\left[ S'|\pa_{x_a}|^{\frac{1}{2}}, w_a\right] h \right\|_{L^2}
\right).
\end{align*}
Using Lemma~\ref{BeHaNa}, one sees that the second term on the right-hand side 
can be bounded by $C\|h\|_{L^2}$. Writing
\begin{align*}
|\pa_{x_a}|^{\frac{1}{2}} 
  \left( w_a^{-1}\overline{g} S^* S f \right)
=&
 \left[|\pa_{x_a}|^{\frac{1}{2}}, w_a^{-1}\overline{g}\right] S^* S f
{}+ 
 w_a^{-1}\overline{g} \left[|\pa_{x_a}|^{\frac{1}{2}}, S^*\right] Sf\\
& {}+
 w_a^{-2}\overline{g}\left[ w_a, S^* |\pa_{x_a}|^{\frac{1}{2}}\right] S f
{}+
 w_a^{-2}\overline{g} S^*
 \left(\left[|\pa_{x_a}|^{\frac{1}{2}}, w_a\right] S f \right)\\
&{}+
 w_a^{-2}\overline{g}S^*
 \left(w_a\left[|\pa_{x_a}|^{\frac{1}{2}}, S \right] f \right)
 +
 w_a^{-2}\overline{g} S^* \left(w_a S|\pa_{x_a}|^{\frac{1}{2}}f \right),
\end{align*}
and using Lemma~\ref{BeHaNa} to estimate the commutators, we obtain
\begin{align*}
 &\left\|
  |\pa_{x_a}|^{\frac{1}{2}} 
  \left( w_a^{-1}\overline{g} S^* S f \right)
 \right\|_{L^2}\\
 &\le 
 C\left(
   \left\|w_a^{-2}g\right\|_{L^\infty}
   +
   \left\|w_a^{-1}\pa_{x_a}g\right\|_{L^\infty}
 \right)
 \left(
  \left\|f \right\|_{L^2}
  +
  \left\|w_a S|\pa_{x_a}|^{\frac{1}{2}}f \right\|_{L^2}
 \right).
\end{align*}
To sum up, we obtain
\begin{align*}
 \bigl| \jbf{S f, S(g \pa_{x_a} h)}_{L^2} \bigr|
 \le&
 C\left(
  \left\| \jbf{\frac{x_a}{\jb{t}}}^{2} g \right\|_{L^{\infty}}
 + 
  \left\| \jbf{\frac{x_a}{\jb{t}}} \pa_{x_a} g \right\|_{L^{\infty}}
 \right)
 \nonumber\\
 &\times
 \left(
  \left\| w_a S|\pa_{x_a}|^{\frac{1}{2}}f \right\|_{L^2} + \|f\|_{L^2}
 \right)\left(
  \left\| w_a S'|\pa_{x_a}|^{\frac{1}{2}}h \right\|_{L^2}  + \|h\|_{L^2}
 \right).
\end{align*}
Similarly, we have 
\begin{align*}
 \bigl| \jbf{S f, S(g \overline{\pa_{x_a} h})}_{L^2} \bigr|
 \le&
 C\left(
  \left\| \jbf{\frac{x_a}{\jb{t}}}^{2} g \right\|_{L^{\infty}}
 + 
  \left\| \jbf{\frac{x_a}{\jb{t}}} \pa_{x_a} g \right\|_{L^{\infty}}
 \right)
 \nonumber\\
 &\times
 \left(
  \left\| w_a S|\pa_{x_a}|^{\frac{1}{2}}f \right\|_{L^2}
  +
  \|f\|_{L^2}
 \right)\left(
  \left\| w_a S'|\pa_{x_a}|^{\frac{1}{2}}h \right\|_{L^2}
  +
  \|h\|_{L^2}
 \right).
\end{align*}
On the other hand, it follows from the relation 
$\frac{x_a}{\jb{t}}=\frac{1}{\jb{t}}J_{m,a} -\frac{it}{m\jb{t}} \pa_{x_a}$ 
and Lemma~\ref{Lem_kla} that 
\begin{align*}
 \left\| \jbf{\frac{x_a}{\jb{t}}}^{2} g \right\|_{L^{\infty}}
 {}+
  \left\| \jbf{\frac{x_a}{\jb{t}}} \pa_{x_a} g \right\|_{L^{\infty}} 
 &\le
 C \sum_{|\alpha|\le 2}\|\Gamma_m^{\alpha} g \|_{L^{\infty}}\\
 &\le 
 \frac{C}{\jb{t}^{\frac{d}{2}}} 
 \sum_{|\alpha|\le [\frac{d}{2}]+3} \|\Gamma_m^{\alpha} g \|_{L^{2}}.
\end{align*}
By piecing them together, we arrive at the desired conclusion.
\qed

\subsection{Proof of Lemma~\ref{Lem_local}}

Before we proceed to the proof of Lemma~\ref{Lem_local}, 
we give several lemmas.
Concerning the uniqueness, using Lemma~\ref{Lem_QLE}, and going similar lines to
the proof of Lemma~\ref{Lem_est_high}, we obtain the following
(observe that $\left[\frac{s}{2}\right]+\left[\frac{d}{2}\right]+2=s$ if
$s=2\left[\frac{d}{2}\right]+3$).

\begin{lem}\label{SchUniqueness} 
Suppose that $d\ge 2$. Let $t_0\in \R$ and $T>0$. 
There is a positive constant $\delta_0$, which is independent of $t_0$ and $T$,
such that if $u$ and $\widetilde{u}$ be solutions to \eqref{Eq}--\eqref{Data2} 
satisfying 
$$
u, \widetilde{u} 
\in 
C\bigl([t_0, t_0+T]; \Sigma^{2\left[\frac{d}{2}\right]+3}(\R^d)\bigr)
 \cap 
 L^\infty\bigl((t_0, t_0+T); \Sigma^{2\left[\frac{d}{2}\right]+4}(\R^d)\bigr),
$$
and if
$$
\sup_{t_0\le t\le t_0+T} \|u(t,\cdot)\|_{\Gamma(t), \left[\frac{d}{2}\right]+4}
\le \delta_0,
$$
then we have
$u(t,x)=\widetilde{u}(t,x)$ for all $(t,x)\in [t_0, t_0+T]\times \R^d$.
\end{lem}

In what follows, for each $j\in \{1,2,3\}$, we put $S_j=S_+(t;1)$ if $m_j>0$,
and $S_j=S_-(t;1)$ if $m_j<0$.
Using Lemma~\ref{Lem_smoothing2} instead of Lemma~\ref{Lem_smoothing}, we can easily modify the proof of Lemma~\ref{Lem_QLE} to obtain
the following (compare \eqref{Est_S_v} with \eqref{QLE21}).

\begin{lem}\label{Lem_QLE2}
Let $d\ge 2$ and $\nu\in (0,1]$. Suppose that $\Phi_{jk,a}$, 
$\Psi_{jk,a}$ and $\Theta=(\Theta_j)_{j=1,2,3}$
be as in Lemma~$\ref{Lem_QLE}$. Let $e_{t_0,T}$ be defined as in 
Lemma~\ref{Lem_QLE} 
with given constants $\lambda_{jk}, \mu_{jk} \in \R\setminus\{0\}$.
Suppose that $v=(v_j)_{j=1,2,3}\in C\bigl([t_0,t_0+T];L^2(\R^d)\bigr)$ satisfies
$$
L_{m_j,\nu} v_j
=\sum_{k=1}^3\sum_{a=1}^d
(\Phi_{jk,a}\pa_{x_a}v_k+\Psi_{jk,a}\pa_{x_a}\overline{v_k})
+\Theta_j.
$$
Then there is a positive constant $\delta$ such that $e_{t_0,T}\le \delta$
implies
\begin{align}
 \sum_{j=1}^3\|S_j(t) v_j(t)\|_{L^2}^2
 \le & 
 \sum_{j=1}^3\|S_j(t_0) v_j(t_0)\|_{L^2}^2 \nonumber\\
 &+
 C\int_{t_0}^t \left(\frac{\|v(\tau)\|_{L^2}^2}{\jb{\tau}}
 +
 \|v(\tau)\|_{L^2}\|\Theta(\tau)\|_{L^2}\right)d\tau
 \label{QLE21}
\end{align}
and
\begin{align}
 \|v(t)\|_{L^2}^2
 \le & 
 C \|v(t_0)\|_{L^2}^2
 +
 C \int_{t_0}^t \left(\frac{\|v(\tau)\|_{L^2}^2}{\jb{\tau}}
 +
 \|v(\tau)\|_{L^2}\|\Theta(\tau)\|_{L^2}\right)d\tau
 \label{QLE22}
\end{align}
for $t\in [t_0, t_0+T]$. Here the positive constants $\delta$ and $C$ depend only on 
$|m_j|$, $|\lambda_{jk}|$ and $|\mu_{jk}|$. 
In particular, they are independent of $\nu\in (0,1]$, $t_0\in \R$, and 
$T$.
\end{lem}

$C_w(I;X)$ denotes the space of $X$-valued weakly continuous functions on an interval $I$, where $X$ is a Hilbert space. 
As in the argument of Chapter 5 in \cite{Rac}, 
we can show the following lemma which we will use 
to obtain the strong continuity of the solution.

\begin{lem}\label{ContiLem}
If $f=(f_j)_{j=1,2,3}\in C_w\bigl([t_0, t_0+T]; L^2(\R^d)\bigr)$ and
$$
\limsup_{t\to t_0+0} \sum_{j=1}^3 \left\|S_j(t)f_j(t)\right\|_{L^2(\R^d)}^2\le \sum_{j=1}^3\left\|S_j(t_0)f_j(t_0)\right\|_{L^2(\R^d)}^2,
$$
then we have
$$
\lim_{t\to t_0+0} \left\|f(t)-f(t_0)\right\|_{L^2(\R^d)}=0.
$$
\end{lem}

\begin{proof}[\underline{Proof of Lemma~\ref{Lem_local}}]
%
We follow the approach of \cite{Chi1}, \cite{Chi2}, \cite{Chi4} 
(see also the appendix of \cite{KT}). 
The proof is divided into three steps:

{\bf Step 1.}
First we consider the auxiliary problem 
\begin{equation}
 \left\{\begin{array}{ll}
L_{m_j,\nu}u_j^\nu
  =
  F_j(u^{\nu},\pa_x u^{\nu}), 
 &t>t_0,\ x \in \R^d,\ j=1,2,3,\\
 u_j^{\nu}(t_0,x)=\psi_j(x), 
 &x\in \R^d,\ j=1,2,3
\end{array}\right.
 \label{eq_aux}
\end{equation}
with $\nu \in (0,1]$. 
Due to the stronger smoothing property of $e^{\nu t \Delta}$, 
we can easily solve it in some interval $[t_0,t_0+T_{\nu}]$ 
by the standard contraction mapping principle.
More precisely, if $\psi=(\psi_j)_{j=1,2,3}\in\Sigma^s(\R^d)$ with some
$s\ge 2\left[\frac{d}{2}\right]+4$, we have the solution
$$
 u^\nu=(u_j^\nu)_{j=1,2,3}\in C\bigl([t_0, t_0+T_\nu]; \Sigma^{s}(\R^d)\bigr)
$$ 
with
$\nabla \Gamma_{m_j}^\alpha u_j^{\nu}\in L^2\bigl((t_0, t_0+T_\nu)\times \R^d\bigr)$
for $j\in\{1,2,3\}$ and $\alpha\in (\Z_+)^{2d}$ 
satisfying $|\alpha| \le s$,
where $T_\nu$ is a positive number depending only on $\nu$ and $\|\psi\|_{\Gamma(t_0), 2\left[\frac{d}{2}\right]+4}$. Observe that we have $\left[\frac{k+1}{2}\right]+\left[\frac{d}{2}\right]+2\le k$ for $k\ge 2\left[\frac{d}{2}\right]+4$, which enables us to estimate $\sup_{t\in [t_0, t_0+T]} \|u(t)\|_{\Gamma, k}$ for $k\ge 2\left[\frac{d}{2}\right]+5$
inductively if we have the estimate for 
$\sup_{t\in [t_0,t_0+T]}\|u(t)\|_{\Gamma, 2\left[\frac{d}{2}\right]+4}$.

{\bf Step 2.}  
Next we show that we can solve \eqref{eq_aux} up to some time which is independent of $\nu$.
For a solution $u^\nu$ to \eqref{eq_aux} in $[t_0, t_0+T]$ and $k\in \Z_+$, 
we put
$$
E_{\nu, k}(T)=\sup_{t\in[t_0, t_0+T]} \|u^{\nu}(t)\|_{\Gamma, k}.
$$
For simplicity of exposition, we write 
$S_j\Gamma_{m_j}^\alpha u_j(t)$ and $(S_j\Gamma_{m_j}^\alpha)(t_0) \psi_j$ 
for 
$S_j(t)\Gamma_{m_j}(t)^\alpha u_j(t)$ and 
$S_j(t_0)\Gamma_{m_j}(t_0)^\alpha \psi_j$, respectively. 
Let $\delta_0$ be the constant coming from Lemma~\ref{SchUniqueness}.
The next lemma is the goal of this step.

\begin{lem}\label{LE_approx} 
{\rm (i)}\ Let $d\ge 2$, $t_0\in\R$, and $B>0$. 
For any $s\ge 2\left[\frac{d}{2}\right]+4$, 
there is a positive constant $\delta_s(\le \delta_0)$,
which is independent of $t_0$ and $B$, such that
for any $\nu \in (0,1]$ and any $\psi\in \Sigma^s(\R^d)$ with
$$
\|\psi\|_{\Gamma(t_0), \left[\frac{d}{2}\right]+4}\le \eps < \delta_s
$$
and $\|\psi\|_{\Gamma(t_0), 2\left[\frac{d}{2}\right]+4}\le B$,
the initial value problem \eqref{eq_aux} admits a unique solution 
$u^\nu\in C\bigl([t_0, t_0+T_s^*]; \Sigma^s(\R^d)\bigr)$
with $\nabla \Gamma_{m_j}^\alpha u_j^\nu\in L^2\bigl((t_0, t_0+T_s^*)\times \R^d\bigr)$ for
any $|\alpha|\le s$ and $j=1,2,3$, where $T_s^*=T_s^*(\eps, B)$ is a positive constant which can be determined only by $s$, $\eps$ and $B$, and is independent of $t_0$ and $\nu$. 
Moreover, there are positive constants $C_s$ and $D_s$, which are independent of $\nu$, such that
\begin{align}
\label{Boundedness00}
E_{\nu, \left[\frac{d}{2}\right]+4}(T_s^*) \le & \frac{\eps+\delta_s}{2}(<\delta_s),\\
\label{Boundedness01}
E_{\nu, s}(T_s^*)\le & C_s,\\
\label{Boundedness02}
\sum_{j=1}^3\left\|S_j\Gamma_{m_j}^\alpha u_j^{\nu}(t)\right\|_{L^2}^2
\le & \sum_{j=1}^3 \left\|(S_j\Gamma_{m_j}^\alpha)(t_0)\psi_j\right\|^2_{L^2}+D_s(t-t_0),\quad |\alpha|\le s
\end{align}
for $t\in [t_0, t_0+T_s^*]$.
\\
{\rm (ii)}\ 
If we replace $m_j$ by $-m_j$ , and $F_j(u,\pa_x u)$ by $-F_j(u,\pa_x u)$ for 
$j=1,2,3$ in \eqref{eq_aux}, then the assertion of {\rm (i)} remains true 
with the same constants.
\end{lem}

\begin{proof}
We put $s_0=\left[\frac{d}{2}\right]+4$ and $s_1=2\left[\frac{d}{2}\right]+4$.
Let $u^\nu$ be the solution to \eqref{eq_aux} for some $T>0$.
We write
$u_j^{\nu, (\alpha)}=\Gamma_{m_j}^\alpha u_j^\nu$
for $|\alpha| \le s$, and $u^{\nu,(\alpha)}=\bigl(u_j^{\nu,(\alpha)}\bigr)_{j=1,2,3}$. 
Then we have
$$
L_{m_j, \nu} u_j^{\nu, (\alpha)}= \Gamma_{m_j}^\alpha \bigl(F_j(u^\nu,\pa_xu^\nu)\bigr)-i\nu[\Delta, \Gamma_{m_j}^\alpha]u^{\nu}_j.
$$
We modify the proof of Lemma~\ref{Lem_est_high}:
Let $g_{j,a}^{\nu, \alpha}$ and $h_{j,a}^{\nu,\alpha}$ be 
given by replacing $u_k$ with $u_k^{\nu}$ 
in the definition of $g_{j,a}^\alpha$ and $h_{j,a}^\alpha$, 
respectively. Let $\lambda_j$ and $\mu_j$ be defined as in the proof of 
Lemma~\ref{Lem_est_high}. Then, similarly to \eqref{Est_coeff1}, we get
$$
 \sum_{|\alpha|\le s}\sum_{j=1}^3\sum_{a=1}^d
 \sum_{|\gamma|\le [\frac{d}{2}]+3} 
 \Bigl(
 \|\Gamma_{\lambda_j}^{\gamma} g_{j, a}^{\nu, \alpha}(t)\|_{L^2}
 +
 \|\Gamma_{\mu_j}^{\gamma} h_{j, a}^{\nu, \alpha}(t)\|_{L^2}
 \Bigr)
 \le 
 C_s^*E_{s_0}(T)
$$
with some positive constant $C_s^*$. Let $\delta$ be the constant from 
Lemma~\ref{Lem_QLE2}. If $C_s^*E_{s_0}(T)\le \delta$, 
then \eqref{QLE22} in Lemma~\ref{Lem_QLE2} leads to
\begin{align}
\|u^{\nu}(t)\|_{\Gamma, k}^2
 \le & 
 A_s\left(
   \|\psi\|_{\Gamma(t_0),k}^2
   + 
   \left( 
      1+E_{\nu, \left[\frac{k}{2}\right]+\left[\frac{d}{2}\right]+2}(T)
   \right)  
   \int_{t_0}^t\|u^\nu(\tau)\|_{\Gamma,k}^2d\tau
 \right)
\label{Ene05}
\end{align}
for $0\le k\le s$ with some positive constant $A_s(\ge 1)$.
We set $F_j^\nu=F_j\bigl(u^\nu, \pa_x u^\nu\bigr)$.
Similarly to the proof of Lemma~\ref{Lem_est_low}, 
the standard energy inequality implies
\begin{align}
 \|u^\nu(t)\|_{\Gamma, s_0}^2
 \le &
 \|\psi(t_0)\|_{\Gamma(t_0), s_0}^2
 +
 2\sum_{j=1}^3\sum_{|\alpha|\le s_0}
  \int_{t_0}^t 
  \left| \jbf{
    \Gamma_{m_j}^\alpha F_j^\nu(\tau), \Gamma_{m_j}^\alpha u_j^\nu(\tau)
   }_{L^2} \right| d\tau
\nonumber\\
\le &
\|\psi(t_0)\|_{\Gamma(t_0), s_0}^2+C'\int_{t_0}^t
\|u^\nu(\tau)\|_{\Gamma, s_1}^2\|u^\nu(\tau)\|_{\Gamma, s_0}d\tau
\label{Ene06}
\end{align}
with a positive constant $C'$, where we have used
$
s_0+1\le s_1$ and $\left[\frac{s_0}{2}\right]+\left[\frac{d}{2}\right]+2\le s_1
$ 
for $d\ge 2$.

We put
$$
\delta_s:=\frac{\delta}{C_s^*}.
$$
Suppose that $\|\psi\|_{\Gamma(t_0), s_0}\le \eps<\delta_s$ and
$\|\psi\|_{\Gamma(t_0), s_1}\le B$. We set 
$$
M_0:=\frac{\eps+\delta_s}{2},\ M_1:=2\left(1+\sqrt{A_s}\right)B,
$$
and
\begin{align*}
 T^*=\sup\left\{ \tau\in [0, T];\, 
  \sup_{t_0\le t\le t_0+\tau} \|u^{\nu}(t)\|_{\Gamma(t),s_0}\le M_0,\, 
  \sup_{t_0\le t\le t_0+\tau} \|u^{\nu}(t)\|_{\Gamma(t),s_1}\le M_1 
 \right\}.
\end{align*}
Since $\|u^{\nu}(t_0)\|_{\Gamma, s_0}\le \eps <M_0$ and $\|u^\nu(t_0)\|_{\Gamma, s_1}\le B<M_1$, we see that $T^*>0$
by the continuity of $\|u^\nu(t)\|_{\Gamma, s}$ for $s=s_0, s_1$ with respect to $t$.
Observing that $C_s^*M_0<C_s^*\delta_s\le \delta$ and 
$\left[\frac{s_1}{2}\right]+\left[\frac{d}{2}\right]+2=s_1$, 
it follows from \eqref{Ene05} with $k=s_1$ that
\begin{align}
 \|u^\nu(t_0+T^*)\|_{\Gamma, s_1}^2
 \le & 
 A_s \left(B^2+\int_{t_0}^{t_0+T^*} (1+M_{1})M_{1}^2d\tau\right)
 \nonumber\\
 \le &
 \frac{M_1^2}{4}+A_s(1+M_1)M_{1}^2T^*.
\label{Ene10}
\end{align}
\eqref{Ene06} yields
\begin{align}
\|u^\nu(t_0+T^*)||_{\Gamma, s_0}^2\le & \eps^2+C'\int_{t_0}^{t_0+T^*} M_0M_1^2d\tau
\le \eps^2+C'M_0M_1^2T^*.
\label{Ene11}
\end{align}
Now we put 
$$
T_s^*:=\min\left\{\frac{1}{2A_s(1+M_{1})}, \frac{M_0^2-\eps^2}{2C'M_0M_1^2}\right\},
$$
and suppose that $T\le T_s^*$. Then we have $T^*=T$, because
\eqref{Ene10} and \eqref{Ene11} imply
\begin{align*}
\|u^\nu(t_0+T^*)\|_{\Gamma, s_1}^2\le \frac{3}{4}M_1^2<M_1^2,\quad
\|u^\nu(t_0+T^*)\|_{\Gamma, s_0}^2\le \frac{\eps^2+M_0^2}{2}<M_0^2. 
\end{align*}

We have proved that, if $\|\psi\|_{\Gamma(t_0),s_0}\le \eps<\delta_s$,
$\|\psi\|_{\Gamma(t_0), s_1}\le B$,
and $T\le T_s^*$, then we have
\begin{equation}
\label{CC01}
\sup_{t_0\le t\le t_0+T} \|u^\nu(t)\|_{\Gamma, s_0}\le M_{0}(<\eps_s)
\end{equation}
and
\begin{equation}
\label{CC02}
\sup_{t_0\le t\le t_0+T} \|u^\nu(t)\|_{\Gamma, s_1}\le M_{1}.
\end{equation}
Therefore we see that there is a unique solution $u^\nu$ to \eqref{eq_aux}
on $[t_0,t_0+T_s^*]\times \R^d$.

\eqref{CC01} implies \eqref{Boundedness00}. 
We are going to prove that, for $s_1\le k\le s$, there is a positive constant
$C_{s,k}$ such that
\begin{equation}
\label{CC03}
\sup_{t_0\le t\le t_0+T_s^*} \|u^\nu(t)\|_{\Gamma,k}\le C_{s,k}.
\end{equation}
If $k=s_1$, \eqref{CC03} follows from \eqref{CC02}. Suppose that
\eqref{CC03} is true for some $k$ with $s_1\le k\le s-1$.
Noting that we have
$\left[\frac{k+1}{2}\right]+\left[ \frac{d}{2} \right]+2\le k$ for $k\ge s_1$, \eqref{Boundedness00} and
\eqref{Ene05} yield
\begin{align*}
\|u^\nu(t)\|_{\Gamma, k+1}^2\le A_s\left(\|\psi\|_{\Gamma(t_0), k+1}^2+
(1+C_{s,k})\int_{t_0}^t \|u^\nu(\tau)\|_{\Gamma, k+1}^2 d\tau\right).
\end{align*}
Now the Gronwall Lemma implies
$$
\sup_{t_0\le t\le t_0+T_s^*} \|u^\nu(t)\|_{\Gamma, k+1}\le \left(A_s\|\psi\|_{\Gamma(t_0),k+1}^2e^{A_s(1+C_{s,k})T_s^*}\right)^{\frac{1}{2}}=:C_{s,k+1},
$$
which shows \eqref{CC03} with $k$ replaced by $k+1$.
\eqref{CC03} with $k=s$ shows \eqref{Boundedness01}.

Finally, in view of \eqref{Boundedness00} and \eqref{Boundedness01}, 
going similar lines to the proof of \eqref{Ene05} but
using \eqref{QLE21} instead of \eqref{QLE22}, we obtain \eqref{Boundedness02}.

Investigating the proof above, we can easily check the assertion (ii).
This completes the proof of Lemma~\ref{LE_approx}.
\end{proof}

{\bf Step 3.} 
We write $I=(t_0, t_0+T_s^*)$ and $\overline{I}=[t_0, t_0+T_s^*]$
for simplicity of exposition.

\eqref{Boundedness01} implies the boundedness of
$\{u^\nu\}$ in $L^\infty(I; \Sigma^s)$, and
the weak* compactness of $L^\infty(I; \Sigma^s)$ implies that
there is a sequence $\{\nu_k\}\subset (0,1]$ with $\lim_{k\to \infty} \nu_k=0$ 
such that $u^{\nu_k}$ converges 
to some function $u$ weakly* in $L^\infty(I; \Sigma^s)$
as $k\to\infty$. By \eqref{Boundedness00}, we get
\begin{equation}
\label{DS3'}
\sup_{t\in I} \|u(t)\|_{\Gamma, s_0}\le \liminf_{k\to\infty} \sup_{t\in I} \bigl\|u^{\nu_k}(t)\bigr\|_{\Gamma, s_0}\le \frac{\eps+\delta_s}{2}(<\delta_s),
\end{equation}
which is nothing but \eqref{DS3}.

\eqref{Boundedness01} and \eqref{eq_aux} show that
$\{u^{\nu_k}\}$ is bounded in $C^{0,1}\bigl(\overline{I}; H^{s-2}\bigr)$,
which, together with the boundedness of $\{u^{\nu_k}\}$ in $L^\infty(I; H^s)$, implies that
$\{u^{\nu_k}\}$ is bounded in $C^{0,\frac{1}{2}}\bigl(\overline{I}; H^{s-1}\bigr)$.
 In view of the Rellich-Kondrachov theorem, we obtain by the abstract version of the Ascoli-Arzel\`a theorem that $\{u^{\nu_k}\}$ converges to $u$
 strongly in $C\bigl(\overline{I}; H^{s-2}_{\rm loc}\bigr)$.
This convergence is strong enough to show the convergence of
$F\bigl(u^{\nu_k}, \pa_x u^{\nu_k}\bigr)$
to $F(u,\pa_x u)$ in the distribution sense, and thus $u$ is the local solution to
\eqref{Eq}--\eqref{Data2}.

From \eqref{Eq}, we see that $\pa_tu\in L^\infty(I; \Sigma^{s-2})$,
which shows $u\in C\bigl(\overline{I}; \Sigma^{s-2}\bigr)$. 
Moreover, since 
$(t,x)\mapsto F_j\bigl(u(t,x),\pa_xu(t,x)\bigr)\in L^\infty(I; \Sigma^{s-1})\subset L^1(I; \Sigma^{s-1})$, 
the well-posedness in $\Sigma^{s-1}$ of the linear Schr\"odinger equations
implies $u\in C\bigl(\overline{I}; \Sigma^{s-1}\bigr)$.
Since $u\in L^\infty(I;\Sigma^s)$, by the weak compactness of the Hilbert space $\Sigma^s(\R^d)$, we see that $u\in C_w\bigl(\overline{I};\Sigma^{s})$.

What is left to show is the strong continuity of $u$ as a $\Sigma^s$-valued function.
Let $\tau>0$ and $|\alpha|\le s$. By \eqref{Boundedness02}, we obtain
\begin{align*}
\sup_{t\in[t_0,t_0+\tau]}\sum_{j=1}^3\left\|S_j\Gamma_{m_j}^\alpha u_j(t)\right\|_{L^2}^2
\le & \liminf_{k\to\infty}\sup_{t\in [t_0,t_0+\tau]}
\sum_{j=1}^3\left\|S_j\Gamma_{m_j}^\alpha u_j^{\nu_k}(t)\right\|_{L^2}^2\\
\le&  \sum_{j=1}^3\left\|S_j\Gamma_{m_j}^\alpha 
u_j(t_0)\right\|_{L^2}^2+D_s\tau.
\end{align*}
Hence we get 
$$
 \limsup_{t\to t_0+0}\sum_{j=1}^3\left\|S_j\Gamma_{m_j}^\alpha u_j(t)\right\|_{L^2}^2\le  \sum_{j=1}^3 \left\|S_j\Gamma_{m_j}^\alpha 
u_j(t_0)\right\|_{L^2}^2,
$$
and Lemma~\ref{ContiLem} yields
\begin{equation}
\label{ContiGoal01}
\lim_{t\to t_0+0} \|u(t)-u(t_0)\|_{\Sigma^s}=0.
\end{equation}

We fix $t_1\in (t_0, t_0+T_s^*)$. We consider the problem
\begin{equation}
\label{eq_aux2}
\begin{cases} 
L_{m_j} v_j=F_j(v, \pa_x v),\\
v_j(t_1)=u_j(t_1).
\end{cases}
\end{equation}
\eqref{DS3'} implies that $\|v(t_1)\|_{\Gamma, \left[\frac{d}{2}\right]+4}< \delta_s$.
Because of what we have proved so far, \eqref{eq_aux2} admits
a solution 
$$
 v\in C\bigl([t_1, t_1+T]; \Sigma^{s-1}(\R^d)\bigr)\cap C_w\bigl([t_1,t_1+T];\Sigma^s(\R^d)\bigr)
$$
for some $T>0$, with 
$$
\lim_{t\to t_1+0} \|v(t)-v(t_1)\|_{\Sigma^s}=0.
$$
We may assume $t_1+T\le t_0+T_s^*$. In view of \eqref{DS3'}, we can apply
Lemma~\ref{SchUniqueness} to conclude that
$v(t,x)=u(t,x)$ for $(t,x)\in [t_1, t_1+T]\times \R^d$.
Hence, recalling \eqref{ContiGoal01}, we find
\begin{equation}
\label{ContiGoal02}
\lim_{t\to t_1+0} \|u(t)-u(t_1)\|_{\Sigma^s}=0,\quad t_1\in [t_0, t_0+T_s^*).
\end{equation}

We fix $t_1\in (t_0, t_0+T_s^*]$, and consider the problem
\begin{equation}
\label{eq_aux3}
\begin{cases} 
L_{-m_j} w_j=-F_j(w, \pa_x w),\\
w_j(-t_1)=u_j(t_1).
\end{cases}
\end{equation}
Note that if we put $\widetilde{u}_j(t,x)=u(-t,x)$ for $-t_0-T_s^*\le t\le -t_0$,
then $\widetilde{u}=(\widetilde{u}_j)_{j=1,2,3}$ is a solution to \eqref{eq_aux3}.
Moreover, we have
$$
\Gamma_{-m_j}(t)^\alpha \widetilde{u}_j(t,x)=\Gamma_{m_j}(-t)^\alpha u_j(-t,x),
\quad t\in [-t_0-T_s^*, -t_0].
$$
In particular, we have 
$\Gamma_{-m_j}(-t_1)^\alpha w_j(-t_1,x)=\Gamma_{m_j}(t_1)^\alpha u_j(t_1, x)$. 
Now, thanks to \eqref{DS3'} again, \eqref{eq_aux3} admits
a solution 
$$
w\in C\bigl([-t_1, -t_1+T']; \Sigma^{s-1}(\R^d)\bigr)\cap C_w\bigl([-t_1,-t_1+T']; \Sigma^s(\R^d)\bigr)
$$
for some $T'>0$, with 
$$
\lim_{t\to -t_1+0} \|w(t)-w(-t_1)\|_{\Sigma^s}=0.
$$
We may assume $-t_1+T'\le -t_0$.
As we have observed, \eqref{DS3'} yields
$$
\sup_{t\in [-t_0-T_s^*,-t_0]} \left(
\sum_{j=1}^3\sum_{|\alpha|\le s} \|\Gamma_{-m_j}(t)^\alpha \widetilde{u}_j(t)\|_{L^2}^2
\right)^{1/2}<\delta_s\le \delta_0, 
$$
and Lemma~\ref{SchUniqueness} leads to $w(t,x)=\widetilde{u}(t,x)=u(-t,x)$ for
$(t,x)\in [-t_1,-t_1+T']\times \R^d$. Hence we get
\begin{equation}
\label{ContiGoal03}
\lim_{t\to t_1-0} \|u(t)-u(t_1)\|_{\Sigma^s}=0,\quad t_1\in (t_0, t_0+T_s^*].
\end{equation}
By \eqref{ContiGoal02} and \eqref{ContiGoal03}, we obtain $u\in C\bigl(\overline{I}; \Sigma^s\bigr)$ as desired.

Because of \eqref{DS3}, uniqueness of the solution follows from 
Lemma~\ref{SchUniqueness}.
\end{proof}

\medskip
\subsection*{Acknowledgments}
The authors are grateful to Professors Hiroyuki Takamura 
and Hideo Kubo for informing the authors of the observation in \cite{A}. 
The authors also thank Professor Nakao Hayashi for his 
encouragement and useful conversations.
The work of S.K. is partially supported by Grant-in-Aid for Scientific Research (C) (No.~23540241), JSPS.
The work of H.S. is partially supported by Grant-in-Aid for 
Scientific Research (C) (No.~25400161), JSPS.


\end{document}